\documentclass[11pt]{article}
\usepackage{amssymb}
\usepackage{amsfonts}
\usepackage{amsmath}
\usepackage{amsthm}
\usepackage[bottom]{footmisc} 
\usepackage{bbm} 
\usepackage{stmaryrd} 
\usepackage{setspace} 
\usepackage[newitem,newenum,neverdecrease]{paralist} 
\usepackage[letterpaper]{geometry} 

\usepackage[style=numeric,backend=bibtex,natbib=true]{biblatex}
\def\be{\beta}

\def\de{\delta}

\def\si{\sigma}

\def\om{\omega}

\def\De{\Delta}

\def\Om{\Omega}

\let\on=\operatorname

\def\R{\mathbb R}

\def\ol{\overline}

\def\wt{\widetilde}

\def\x{\times}
\def\one{{\mathbbm 1}}

\def\cA{{\mathcal A}}

\def\fC{{\mathfrak C}}
\def\bD{{\mathbb D}}
\def\cD{{\mathcal D}}
\def\bE{{\on{\,\mathbb E}}}
\def\cF{{\mathcal F}}

\def\cG{{\mathcal G}}

\def\bH{{\mathbb H}}

\def\bN{{\mathbb N}}
\def\bS{{\mathbb S}}
\def\bU{{\mathbb U}}
\def\bX{{\mathbb X}}
\def\bY{{\mathbb Y}}
\def\rd{\mathrm{d}}
\def\cP{\mathcal P}
\def\bP{\mathbb P}
\def\bQ{\mathbb Q}
\def\cS{\mathcal S}
\def\bZ{\mathbb Z}
\def\po{\text{p.o.}} 
\def\se{\text{se.}}  

\theoremstyle{plain}
\newtheorem{theorem}{Theorem}
\newtheorem{proposition}{Proposition}
\newtheorem{lemma}{Lemma}
\newtheorem{assumption}{Assumption}

\newtheorem{definition}{Definition}

\theoremstyle{definition}
\newtheorem{remark}{Remark}
\newtheorem*{notation*}{Notation}
\newtheorem*{remark*}{Remark}
\newtheorem{example}{Example}

\usepackage{hyperref}

\let\oldautoref\autoref
\makeatletter
\renewcommand\autoref[1]{\@first@ref#1,@}
\def\@throw@dot#1.#2@{#1}
\def\@set@refname#1{
    \edef\@tmp{\getrefbykeydefault{#1}{anchor}{}}%
    \def\@refname{\@nameuse{\expandafter\@throw@dot\@tmp.@autorefname}s}%
}
\def\@first@ref#1,#2{%
  \ifx#2@\oldautoref{#1}\let\@secondref\@gobble
  \else%
    \@set@refname{#1}
    \@refname~\ref{#1}
    \let\@secondref\@second@ref
  \fi%
  \@secondref#2%
}
\def\@second@ref#1,#2{%
  \ifx#2@ and~\ref{#1}\let\@nextref\@gobble
  \else, \ref{#1}
    \let\@nextref\@next@ref
  \fi%
  \@nextref#2%
}
\def\@next@ref#1,#2{%
   \ifx#2@, and~\ref{#1}\let\@nextref\@gobble
   \else, \ref{#1}
   \fi%
   \@nextref#2%
}
\makeatother

\let\oldtheequation\theequation
\makeatletter
\def\tagform@#1{\maketag@@@{\ignorespaces#1\unskip\@@italiccorr}}
\renewcommand{\theequation}{(\oldtheequation)}
\makeatother

\begin{document}

\title{%
Two-Armed Restless Bandits with Imperfect Information: 
Stochastic Control and Indexability%
\thanks{We are grateful to Richard Holden, Peter Michor, Derek Neal, Ariel Pakes, Yuliy Sannikov, Mete Soner, Josef Teichmann, and seminar participants at Barcelona GSE and Harvard University for helpful comments. Financial support from the Education Innovation Laboratory at Harvard University is gratefully acknowledged. Correspondence can be addressed to the authors by e-mail: \url{rfryer@fas.harvard.edu} [Fryer] or \url{philipp.harms@stochastik.uni-freiburg.de} [Harms]. The usual caveat applies.}\bigskip}
\author{Roland G. Fryer, Jr.\\Harvard University and NBER
\and Philipp Harms\\Freiburg University}
\date{March 2017}
\maketitle

\begin{abstract} 
We present a two-armed bandit model of decision making under uncertainty where the expected return to investing in the ``risky arm'' increases when choosing that arm and decreases when choosing the ``safe'' arm. These dynamics are natural in applications such as human capital development, job search, and occupational choice. Using new insights from stochastic control, along with a monotonicity condition on the payoff dynamics, we show that optimal strategies in our model are stopping rules that can be characterized by an index which formally coincides with Gittins' index. Our result implies the indexability of a new class of restless bandit models.
\end{abstract}

\pagebreak

\section{Introduction.}\label{sec:introduction}

Bandit models are decision problems where, at each instant of time, 
a resource like time, effort, or money has to be allocated strategically between several options, referred to as the arms of the bandit. When selected, the arms yield payoffs that typically depend on unknown parameters. Arms that are not selected remain unchanged and yield no payoff. The key idea in this class of models is that agents face a tradeoff between experimentation (gathering information on the returns to each arm) and exploitation (choosing the arm with the highest expected value).

Over the past sixty years, bandit models have become an important framework in economic theory, applied mathematics and probability, and operations research. They have been used to analyze problems as diverse as market pricing, the optimal design of clinical trials, product search and the research and development activities of firms (\citet{Rothschild,Berry,BoltonHarris1999,KellerRady2010}). To understand how firms set prices without a clear understanding of their demand curves, \citet{Rothschild} posits that firms repeatedly charge prices and observe the resulting demand. Setting prices too high or too low is costly for firms (experimentation), but allows them to learn about the optimal price (exploitation). In the optimal design of clinical trials, \citet{Berry} formulate the problem as: given a fixed research budget, how does one allocate effort among competing projects, whose properties are only partially known at a given point in time but may be better understood as time passes. In product search, customers sample products to learn about their quality. Their optimizing behavior can be described as in \citet{BoltonHarris1999,BoltonHarris2000}. In these models, news about the quality of the product arrive continuously. The situation where news arrive only occasionally, e.g.\@ in the form of break-throughs in research, is modeled by Keller et al. \citep{KellerRadyCripps,KellerRady2010}.

An important assumption in the classical bandit literature is that the reward distribution of arms that are not chosen does not evolve; they rest (\citet{Gittins2011}). This assumption seems natural in many applications. Yet, in many other important scenarios, it seems overly restrictive.\footnote{The importance of relaxing this assumption has been recognized early on in the seminal work of \citet{Whittle1988}, who proposed clinical trials, aircraft surveillance, and assignment of workers to tasks as potential applications.} Consider, for instance, the possibility of dynamic complementarities in human capital production.\footnote{\citet{CunhaHeckman} make a similar argument in a different context.} Imagine a student who has the choice of whether or not to invest effort into her school work. Today's effort is rewarded by being more at ease with tomorrow's course work, or the ability to glean a deeper understanding from class lectures. As \citet{CunhaHeckman2010} note, ``learning begets learning.'' Conversely, not doing one's assignments today might give instantaneous gratification, but makes tomorrow's school work harder. More generally, this dynamic can be found in the context of human capital formation when early investments in human capital increase the expected payoff of future investments, while a lack of early investments has the reverse effect. These dynamics require arms that evolve even when they are not used.

As a second example, consider an unemployed worker looking for a job. With every job application, she gathers both information about the job market and experience in the application process, which typically increases her chances of successful future job applications. Conversely, not actively searching for a job may decrease the probability of finding a job in future applications. This is empirically well-documented (\citet{kroft2012duration}) and could be due to market penalties for unemployment spells, being disconnected from the changing characteristics of the job market and the application process, or be considered a signal of low motivation by potential employers. 

Bandits whose inactive arms are allowed to evolve are known as restless bandits.\footnote{Bandits where the active and passive action have opposite effects on payoffs are called \emph{bi-directional bandits} (\citet{Glazebrook2006b}), and our model falls into this class.} Generally, optimal strategies for restless bandits are unknown.\footnote{Numerical solutions can be obtained by (possibly approximate) dynamic programming or a linear programming reformulation of the problem (\citet{Kushner,Powell2007,NinoMora2001}).} Nevertheless, when a certain indexability condition is met, Whittle's index \cite{Whittle1988} can lead to approximately optimal solutions (\citet{Weber1990,Weber1991}). This index plays the same fundamental role for restless bandits that Gittins' index \cite{Gittins1979} has for classical ones: it decomposes the task of solving multi-armed bandits into multiple tasks of solving bandits with one safe and one risky arm. The safe arm yields constant rewards and can be interpreted as a cost of investment in the risky arm. Deriving conditions that identify general classes of indexable restless bandit models is an important contribution---permitting more complete analysis of decision problems in which choices jointly effect instantaneous payoffs as well as the distribution of those payoffs in the future---and the subject of this paper.

The origins of this work are the classical bandit models of \citet{BoltonHarris1999}, \citet{KellerRady2010}, and \citet{CohenSolan}, that we extend to the restless case. In these  works, the reward from the risky arm is Brownian motion, a Poisson process, or a Levy process. The unobserved quantity is a Bernoulli variable. Our model is an extension of these models containing them as special cases.\footnote{However, some of these works focus on strategic equilibria involving multiple agents, whereas we only treat the single agent case.} Namely, we allow the same generality of reward processes with both volatility and jumps, but make the reward distribution dependent on the type of the agent and the history of past investments. The latter dependence is mediated by a real valued variable that increases while the agent invests in the risky arm and decreases otherwise. In line with our motivating examples of human capital formation and job search, we call this variable the agent's human capital. 

The bandit model is first formulated as a problem of stochastic optimal control under \emph{partial observations} in continuous time.\footnote{Modeling time as continuous allows one to treat discrete-time models with varying step sizes in a unified framework. We show in \autoref{thm:equiv} that discrete-time versions of the model converge to the continuous-time limit. This is not true in some other and recent approaches (see \autoref{rem:wel,rem:fil}).} Standard formulations of the control problem with partial observations do not work for restless bandit models (see \autoref{sec:lit:sep} for a discussion). However, we show that the frameworks of \citet{FlemingNisio}, \citet{Wonham}, and \citet{Kohlmann} can be used and extended to general controlled Markov processes. We describe these issues in detail in \autoref{sec:lit:sep}, since they are rarely discussed in the context of bandit problems. 

The first result in this paper is a \emph{separation theorem} (\autoref{thm:equiv}) that establishes the equivalence of the  control problem with partial observations to a control problem with full observations called the separated control problem. This equivalence is crucial for the solution of the problem and is implicitly used in many works, including \citet{BoltonHarris1999}, \citet{KellerRady2010}, and \citet{CohenSolan}. The separated problem is derived from the partially observable one by replacing the unobserved quantity by its filter, which is its conditional distribution given the past observations. Put differently, the filter is the belief of the agent about the hidden state variable. In the separated problem, admissibility of controls is defined without the strong measurability constraints present in the control problem with partial observations. Therefore, standard results about the existence of optimal controls and the equivalence to dynamic and linear programming can be applied.  

Our second, and main, result (\autoref{thm:main}) is the \emph{optimality of stopping rules}, meaning that it is always better to invest first in the risky arm and then in the safe arm instead of the other way round. This result hinges on the  monotonic dependence of payoffs on past investment. Intuitively, the sequence of investments matters for two reasons. First, investments in the risky arm reveal information about the distribution of future rewards. The sooner this information becomes available, the better. Second, early investments in the safe arm deteriorate the rewards of later investments in the risky arm. By contrast, early investments in the risky arm do not make the safe  arm any less profitable. 

We present an unconventional approach to show the optimality of stopping rules. The work horse of most of the bandit literature is either the Hamilton-Jacobi-Bellman (HJB) equation or a setup using time changes. The inclusion of human capital as a state variable turns the HJB equation into a second order partial differential-difference equation. It seems unlikely that explicit solutions of this equation can be found. Moreover, the approach using time changes is not well adapted to the new dynamics of our model. We circumvent these difficulties by investigating the sample paths of optimal strategies. More specifically, we discretize the problem in time and show that any optimal strategy can be modified such that the agent never invests after a period of not investing and such that the modified strategy is still optimal. This interchange argument has been originally developed by \citet{Berry} for classical bandits. It turns out that the monotonic dependence of the payoffs on the amount of past investment is exactly what is needed to generalize the argument to restless bandits.

Once the optimality of stopping rules is established, it follows easily that optimal strategies can be characterized by an \emph{index} rule. Formally, the index is the same as the one proposed in the celebrated result by \citet{Gittins1979} on classical bandits,  but inactive arms are allowed to evolve. The explicit formula for the index yields comparative statics of optimal strategies with respect to the parameters of the model. Most importantly, subsidies of the safe arm enlarge the set of states where the safe arm is optimal,  which means that our bandit model is \emph{indexable} in the sense of \citet{Whittle1988} (see \autoref{prop:indexability}). More generally, any arm of a multi-armed restless bandit that satisfies our monotonicity condition is indexable. To our knowledge, this is the first time that a sufficient condition for indexability of a general class of restless bandits with continuous state space and a corresponding rich class of reward processes has been formulated.\footnote{\label{foo:indexable}Some sensor management models are indexable and have a continuous state space after their transformation to fully observed Markov decision problems (\citet{Washburn}). This is, however, not the case in their formulation as control problems with partial observations.}

To explain the structure of optimal strategies, we consider how information is processed by agents in our model. We work in a Bayesian setting where the agent has a prior about being either ``high'' or ``low type.'' Rewards obtained from the risky arm depend on this type and are used by the agent to form a posterior belief. The current levels of belief and human capital determine at each stage whether it is optimal to invest in the risky or safe arm. Namely, there is a curve in the belief--human capital domain such that it is optimal to invest in the risky arm if the current level of belief and human capital lies to the right and above the curve. Otherwise, it is optimal to invest in the safe arm. The curve is called the \emph{decision frontier} (see \autoref{prop:frontier}). 

There is, however, an important, and potentially empirically relevant, difference to classical bandit models: not only is the safe arm absorbing---it is depreciating; agents drift further and further away from the frontier. Empirically, this implies that there are very few ``marginal'' agents, i.e., agents at the decision frontier. Programs (e.g. lower class size, school choice, financial incentives) designed to increase student achievement at the margin are likely to be ineffective unless: (a) they are initiated when students get close to the decision frontier, or (b) force inframarginal students to invest in the risky arm (e.g. some charter schools, see \citet{DobbieFryer}). Consistent with \citet{CunhaHeckman}, our model predicts that, on average, the longer society waits to invest, the more aggressive the investment needs to be. This is in stark contrast to classical bandit models, where agents accumulate at or near the frontier (in the sense of \autoref{prop:pop}), and is one of the key motivations of our model.


The paper is structured as follows. 
\autoref{sec:lit} provides a brief review of the bandit literature in economics and applied mathematics.
\autoref{sec:sto} contains the definitions of the control problems and the separation theorem. 
\autoref{sec:res} specializes the general framework of the previous section to restless bandit models satisfying the monotonicity condition and and analyzes the structure of optimal strategies. 
Finally, \autoref{sec:con} concludes.

\section{Previous literature.}\label{sec:lit}

\subsection{Bandit models.}\label{sec:lit:ban}

Originally developed by \citet{Robbins}, bandit models have been used to analyze a wide range of economic and applied math problems.\footnote{\citet{Ghosh}, \citet{Bergemann}, and \citet{Mahajan} provide excellent surveys of the literature on bandit models. The monographs by \citet{PresmanSonin}, \citet{Berry} and \citet{Gittins2011} contain more detailed presentations.} The first paper where a bandit model was used in an economic context is \citet{Rothschild}, in which a single firm facing a market with unknown demand has to determine optimal prices. Subsequent applications of bandit models include partner search, effort allocation in research, clinical trials, network scheduling and voting in repeated elections (\citet{McCall,Weitzman,Berry,Li,Banks1992}).

Classical bandits with reward processes driven by Brownian motion or a Poisson process were first solved by \citet{Karatzas} and \citet{Presman}. Subsequently, \citet{BoltonHarris1999,BoltonHarris2000} and Keller e.a. \cite{KellerRadyCripps,KellerRady2010,KellerRady2015} derived explicit formulas for optimal strategies in the case where the unobservable quantity is a Bernoulli variable and treated strategic interactions of multiple agents. \citet{CohenSolan} unified the formulas obtained for the single agent case and solved a bandit model where the reward is driven by a Levy process with unknown Levy triplet.

Many extensions and variations of classical bandit problems have been proposed, including: bandits with a varying finite or infinite numbers of arms (\citet{Whittle1981,Banks1992}), bandits where an adversary has control over the payoffs (\citet{Auer}), bandits with dependent arms (\citet{Pandey}), bandits where multiple arms can be chosen at the same time (\citet{Whittle1988}), bandits whose arms yield rewards even when they are inactive (\citet{Glazebrook2006b}), and bandits with switching costs (\citet{Banks1994}).

One of the most mathematically challenging extensions is to allow inactive arms to evolve. Such bandits are often referred to as ``restless bandits.''\footnote{Some bandits with switching costs can be modeled as restless bandits (\citet{Jun}).} This term was coined in the seminal paper of \citet{Whittle1988}. Beyond mathematical intrigue, there are many practical applications: aircraft surveillance, sensor scheduling, queue management, clinical trials, assignment of workers to tasks, robotics, and target tracking (\citet{Ny,Veatch,Whittle1988,Faihe,LaScala}). In aircraft surveillance, \citet{Ny} discuss the problem of surveying ships for possible bilge water dumping. A group of unmanned aerial vehicles can be sent to the sites of the ships. The rewards are associated with the detection of a dumping event. The problem falls into the class of sensor management problems, where a set of sensors has to be assigned to a larger set of channels whose state evolves stochastically. In linear Gaussian settings these problems can be reduced to deterministic problems and turn out to be indexable (\citet{ny2011scheduling}). In queue management, \citet{Veatch} consider the task of scheduling a make-to-stock production facility with multiple products. Finished products are stored in an inventory. Too small an inventory risks incurring backorder or lost sales costs, while too large an inventory increases holding costs. In robotics, \citet{Faihe} consider the behaviors coordination problem in a setting of reinforcement learning: a robot is trained to perform complex actions that are synthesized from elementary ones by giving it feedback about its success.

\subsection{Optimal control with partial observations.}\label{sec:lit:sep}

In control problems with partial observations, strategies are not allowed to depend on the hidden state. To enforce this constraint, one requires them to be measurable with respect to the sigma algebra generated by the observations. In continuous time, this measurability condition is not strong enough to exclude pathological cases like \autoref{ex:pat} in this paper. 

This problem was solved in a setting with additive, diffusive noise by requiring the existence of a change of measure, called Zakai's transform (\citet{FlemingPardoux}), which transforms the observation process into standard Brownian motion. Unfortunately, this approach is not amenable to bandit models, where such a change of measure does not exist because the volatility of the observation process depends on the strategy. Another approach, which was applied successfully to classical bandit models, is to define strategies as time changes (\citet{ElKaroui1994}). Unfortunately, this technique does not work for restless bandit problems, where inactive arms are allowed to evolve. 

Our approach can be seen as a generalization of \citet{FlemingNisio,Wonham,Kohlmann}. In these works, the strategies are required to be Lipschitz continuous to ensure well-posedness of the corresponding martingale problem. This excludes discontinuous strategies like cut-off rules, which are typically encountered in bandit problems. We replace the Lipschitz condition by the weaker and more direct requirement that the martingale problem is well-posed. The resulting class of admissible strategies is large enough to contain optimal strategies of classical bandit models and of the restless bandit model in \autoref{sec:res}. It is also small enough to exclude degeneracies like \autoref{ex:pat} and to admit approximations in value by piecewise constant controls (see \autoref{thm:equiv}). For piecewise constant controls the definition of admissibility is unproblematic. 

\subsection{Optimality of stopping rules.}

For classical bandit models with one safe and one risky arm, the optimality of stopping rules is a well-known result (\citet{Berry,ElKaroui1994}). Several approaches to establish this result can be found in the literature.	In one approach, the rewards of each arm are fixed in advance and strategies are time changes. The reward that is obtained under a strategy is the time change applied to the reward process. This setup, which has been proposed by \citet{Mandelbaum}, allows a very simple formulation of the measurability constraints on the strategies. It is, however, not well-suited to bandits with evolving arms.	In a second approach, one solves the Hamilton-Jacobi-Bellman (HJB) equation for the value function. When this succeeds, the explicit form of the value function can be used to establish the optimality of stopping rules (\citet{BoltonHarris1999,KellerRadyCripps,CohenSolan}). In our model, however, the dynamics of the reward distribution introduce an additional state variable, which turns the HJB equation into a non-local partial differential equation which we cannot solve directly. Moreover, the value function might not be a solution in a classical sense. \citet{Pham1995,Pham1998} showed that under suitable assumptions, the value function is a viscosity solution of the HJB equation. It remains open how this could be used to show that stopping rules are optimal. The third approach is to rewrite the problem as a linear programming problem. This makes both classical and restless bandit problems amenable to efficient numerical computations and can also yield some qualitative insight (\citet{NinoMora2001}).\footnote{Another numerical approach is dynamic programming/value function iteration.} The fourth approach (and the one we emulate) is based on a direct investigation of the sample paths of optimal strategies and an evaluation of the benefits of investing in the risky arm sooner rather than later. While this interchange argument was originally developed by \citet{Berry} for classical bandit models, it turns out that the monotonicity assumption on the payoffs is what is needed to make the argument work in the more general setting of restless bandits. 

\subsection{Indexability.}

\citet{Gittins1979} characterized optimal strategies in classical bandit models by an index that is assigned to each arm of the bandit at each instant of time. The optimal strategy is to always choose the arm with the highest index. The indices can be calculated for each arm separately, which reduces the complexity of multi-armed bandits to that of two-armed bandits with one safe and one risky arm.

In general, optimal strategies in restless bandit models do not admit an index representation. Nevertheless, a Lagrangian relaxation of the problem proposed by \citet{Whittle1988} yields index strategies that are approximately optimal (\citet{Weber1990,Weber1991}). The corresponding ``Whittle index'' (\citet{Whittle1988}) is the Lagrange multiplier in a constrained optimization problem and has an economic interpretation as a subsidy for passivity or a fair charge for operating the arm. A major challenge to the deployment of Whittle's index is that it can only be defined when a certain indexability condition is met. In this condition, each arm of the restless bandit is compared to a hypothetical arm with known and constant reward. The indexability condition holds if the set of states where the safe arm is optimal is increasing in the reward from the safe arm.\footnote{This is a monotonicity condition on the optimal strategy, which is not to be confounded with our monotonicity condition on the payoffs and the evolution of human capital.}

The question of indexability of restless bandit models is subtle and not yet fully understood. \citet{Gittins2011} give an overview of various approaches to establish the 	indexability of restless bandit models. Partial answers are known for bandits with finite or countable state spaces. Indexability of such models can be tested numerically in a linear programming reformulation of the Markov decision problem (\citet{Klimov1975}). In another line of research, \citet{NinoMora2001} showed that indexability holds for restless bandits satisfying a partial conservation law, which can be verified by running an algorithm. While this can be used to test the indexability of specific restless bandit problems, it does not provide much qualitative insight into which restless bandits are indexable. One would like to have conditions that identify general classes of indexable restless bandit models---this is the subject of this paper.

Some results in this direction have been obtained for various bandit models related to sensor management and dynamic multichannel access, see the papers of \citet{nino2008index}, \citet{liu2010indexability}, \citet{ny2011scheduling} and the survey of \citet{Washburn}. Further classes of indexable problems are the dual speed problem of \citet{Glazebrook2002}, the maintenance models of \citet{Glazebrook2006}, and the spinning plates and squad models of \citet{Glazebrook2006b}. Our paper is in line with these works in that it trades indexability for specific structural conditions.

\section{Stochastic control with partial observations.}\label{sec:sto}

\autoref{sec:set} provides the general setup. The control problem is formulated in Sections \ref{sec:po}--\ref{sec:se}. \autoref{sec:ass} contains all assumptions and \autoref{sec:equ} the main result. Some general notation can be found in \autoref{sec:app:not} in the Appendix.

\subsection{Setup.}\label{sec:set}

$\bU$ is a finite set, $\bX=\{0,1\}$, and $\bY$ is a finite dimensional vector space.\footnote{Our proofs can be generalized to finite state spaces $\bX$ at the cost of heavier notation and to compact control spaces $\bU$ at the cost of additional criteria ensuring the existence of optimal non-relaxed controls for the discretized separated problem (see e.g. the discussion after Theorem~1.21 in \citet{Seierstad}).} Controls are $\bU$-valued c\`agl\`ad processes $U$, the hidden state is an $\bX$-valued random variable $X$, and the observations are c\`adl\`ag $\bY$-valued processes $Y$. The rewards at time $t$ are given by $b(U_t,X,Y_t)$ for some measurable function $b\colon \bU \x \bX\x\bY \to \R$. Rewards are discounted exponentially at rate $\rho>0$ over an infinite horizon, and the aim is to maximize expected discounted rewards. 

The evolution of $Y$ depends on a c\`agl\`ad $\bU$-valued process $U$ and on the hidden state $X$. More specifically, the joint distribution of $X$ and $Y$ will be characterized by a controlled martingale problem associated to a linear operator
\begin{equation*}
\cA\colon \cD(\cA) 
\subseteq B(\bX\x\bY)
\to B(\bU \x \bX \x \bY),
\end{equation*}
where $B$ denotes the bounded measurable functions. The posterior probability that $X=1$ given $\{\cF^Y_t\}$ is denoted by $P$, i.e., $P$ is a $[0,1]$-valued c\`adl\`ag version of the martingale $\bE[X\mid\cF^Y_t]$. Mathematically speaking, $P$ is called filter of $X$, and economically speaking, the agent's belief in $X=1$. The joint evolution of $(P,Y)$ will be characterized by a linear operator
\begin{equation*}
\cG\colon \cD(\cG) 
\subseteq B([0,1]\x \bY)
\to B(\bU \x [0,1] \x \bY).
\end{equation*}
More specific assumptions on $\cA$, $\cG$, and the payoff function $b$ will be made in \autoref{sec:ass}.

\subsection{Control problem with partial observations.}\label{sec:po}

Our definition of controls with partial observations is non-standard and an improvement over the previous literature, as explained in \autoref{sec:lit:sep}. 

\begin{definition}[Martingale problem for $(\cA,F)$]\label{def:mgp_A}
Let $F$ be a c\`agl\`ad adapted $\bU$-valued process on Skorokhod space $D_{\bY}[0,\infty)$ with its natural filtration. $(X,Y)$ is a solution of the martingale problem for $(\cA,F)$ if there exists a filtration $\{\cF_t\}$, such that $X$ is an $\cF_0$-measurable $\bX$-valued random variable, $Y$ is an $\{\cF_t\}$-adapted c\`adl\`ag $\bY$-valued process, and for each $f \in \cD(\cA)$, 
\begin{equation*}
f(X,Y_t)-f(X,Y_0)
-\int_0^t\cA f(F(Y)_s,X,Y_s) \rd s
\end{equation*}
is an $\{\cF_t\}$-martingale. The martingale problem is called well-posed if existence and local uniqueness holds under the conditions $X=x$ and $Y_0=y$, for all $x\in \bX$ and $y\in\bY$.\footnote{For reference, existence and local uniqueness of the above martingale problem are defined in \autoref{sec:app:mgp} in the Appendix.}
\end{definition}

\begin{definition}[Control with partial observations]\label{def:po:co}
A tuple $(U,X,Y)$ is called a control with partial observations if $U=F(Y)$ holds for some process $F$ as in \autoref{def:mgp_A}, the martingale problem for $(\cA,F)$ is well-posed, and $(X,Y)$ solves the martingale problem for $(\cA,F)$.
\end{definition} 

\begin{definition}[Value of controls with partial observations]
\label{def:po:va}
The value of a control $(U,X,Y)$ with partial observations is defined as
\begin{equation*}
    J^{\po}(U,X,Y)=\bE\left[\int_0^\infty \rho e^{-\rho t} b(U_t,X,Y_t) \rd t\right].
\end{equation*}
The set of controls with partial observations satisfying $\bE[X]=p$ and $Y_0=y$ is denoted by $\fC^\po_{p,y}$. The value function for the control problem with partial observations is
\begin{equation*}
    V^\po(p,y)=\sup \left\{ J^\po(U,X,Y): 
    (U,X,Y) \in \fC^\po_{p,y}\right\}.
\end{equation*}
\end{definition}

\begin{remark}[Well-posedness condition]\label{rem:wel}
Every c\`agl\`ad $\{\cF^Y_t\}$-adapted process $U$ coincides up to a null set with $F(Y)$ for some process $F$ as in \autoref{def:mgp_A} (see \citet{delzeith2004skorohod}). Well-posedness of the martingale problem for $(\cA,F)$ is, however, a much stronger condition. From the agent's perspective, it requires the control to uniquely determine the outcome. From a mathematical perspective, it excludes pathological cases like the one presented in \autoref{ex:pat} below. It also ensures that controls can be approximated in value by piecewise constant controls, where such degeneracies cannot occur (see \autoref{thm:equiv}). 
\end{remark}

\begin{example}[Degeneracy in continuous time]\label{ex:pat}
Let $\bX=\bU=\{0,1\}$, $\bY=\R$, $\cA f(u,x,y)=u(2x-1)f_y(x,y)+\frac12 uf_{yy}(x,y)$ for each $f\in\mathcal D(A) = C^2_b(\bX\x\bY)$. The aim is to maximize $\bE[\int_0^t \rho e^{-\rho t} \rd Y_t]=\bE[\int_0^\infty \rho e^{-\rho t} b(U_t,X,Y_t) \rd t]$ over controls $(U,X,Y)$ of the problem with partial observations, where $b(u,x,y)=u(2x-1)$. The following tuple $(U,X,Y)$ satisfies all conditions of \autoref{def:po:co} except for the well-posedness condition: $X$ is a Bernoulli variable, $W$ is Brownian motion independent of $X$, $Y_t=(t+W_t)X$, $U_t=\one_{(0,\infty)}(t) X$, $F(Y)_t=\one_{(0,\infty)}([Y,Y]_t)$. Nevertheless, $U$ depends on the supposedly unobservable state $X$. Actually, $(U,X,Y)$ is optimal for the control problem with observable $X$, and should not be admitted as a control for the problem with unobservable $X$. 
\end{example}

\begin{remark}[Topology on the set of controls]
So far, there is no topology on the set of controls with partial observations. To get existence of optimal controls, one typically relaxes the control problem by allowing measure-valued controls and shows that the resulting set of admissible controls is compact under some weak topology \citep{kurtz1987martingale,ElKaroui1988}. In control problems with partial observations involving strong admissibility conditions as in \autoref{def:po:co}, the difficulty is that the set of admissible controls is not weakly closed. This difficulty can be avoided by transforming the problem into a standard problem with full observations, i.e., the separated problem. 
\end{remark}

\subsection{Separated control problem.}\label{sec:se} 

The following definition is fully standard \cite{kurtz1987martingale, KurtzStockbridge1998, oksendal2005applied}. 

\begin{definition}[Separated controls]\label{def:se:co}
A tuple $(U,P,Y)$ is called a separated control if there exists a filtration $\{\cF_t\}$ such that $U$ is an adapted, c\`agl\`ad $\bU$-valued process, $(P,Y)$ is an adapted, c\`adl\`ag $[0,1]\x\bY$-valued process, and for each $f \in \cD(\cG)$, the following process is an $\{\cF_t\}$-martingale:
\begin{equation*}
f(P_t,Y_t)-f(P_0,Y_0)-\int_0^t \cG f(U_s,P_s,Y_s) \rd s.
\end{equation*}
\end{definition} 

\begin{definition}[Value of separated controls]\label{def:se:va}
The value of a separated control $(U,P,Y)$ is
\begin{equation*}
    J^{\se}(U,P,Y)
    =\bE\left[\int_0^\infty \rho e^{-\rho t} \ol b(U_t,P_t,Y_t) \rd t\right],
\end{equation*}
where $\ol b(u,p,y)=p b(u,1,y)+(1-p) b(u,0,y)$. The set of controls $\fC^\se_{p,y}$ and the value function $V^\se(p,y)$ are defined similarly as in \autoref{def:po:co}.
\end{definition} 

\begin{remark}[Filtered martingale problem]\label{rem:fil}
Following \citet{Stockbridge}, one could try the alternative approach of defining separated controls as solutions of the filtered martingale problem for $\cA$, i.e., the process 
\begin{equation*}
\Pi_t(\rd x)=P_t \delta_1(\rd x)+(1-P_t)\delta_0(\rd x)
\end{equation*}
is $\{\cF^Y_t\}$-adapted and for each $f\in\cD(\cA)$, the process
\begin{equation*}
\int_{\bX} f(x,Y_t) \Pi_t(\rd x) - \int_{\bX} f(x,Y_0) \Pi_0(\rd x) 
- \int_0^t \int_{\bX} \cA f(U_s,x,Y_s) \Pi_s(\rd x) \rd s
\end{equation*}
is a martingale with respect to some filtration containing $\{\cF^Y_t\}$. Unfortunately, this definition does not rule out the pathological control presented in \autoref{ex:pat}, and cannot be used for this reason.
\end{remark}

\begin{remark}[Topology on the set of controls]\label{rem:se:top}
The set of separated controls can be topologized by regarding them as probability measures on the canonical space $L_\bU[0,\infty)\x D_{[0,1]\x\bY}[0,\infty)$, subject to the condition that the coordinate process solves the martingale problem in \autoref{def:se:co}. Compactness and existence of optimal controls can be obtained by relaxing the control problem. This amounts to replacing $L_\bU[0,\infty)$ by the space of measures on $\bU\times[0,\infty)$ with $[0,\infty)$-marginal equal to the Lebesgue measure and endowing this space with the vague topology \citep{jacod1981type, ElKaroui1988}. It should be noted, however, that relaxed separated controls are not filters of relaxed controls with partial observations (see \autoref{sec:app:non} in the Appendix). In other words, filtering is a non-linear operation on control problems, which does not commute with relaxation. 
\end{remark}

\subsection{Specification of the generators and assumptions.}\label{sec:ass}

We specify the operators $\cA$ and $\cG$ in a list of assumptions (Assumptions~\ref{ass:A}--\ref{ass:wel2}). Assumptions~\ref{ass:A}--\ref{ass:bigjumps} are unproblematic because they are direct conditions on the model coefficients and can be satisfied by definition. \autoref{ass:wel,ass:wel2} are more difficult to verify. They require well-posedness of certain martingale problems related to $\cA$ and $\cG$. This can be checked using standard results \cite{komatsu1973markov,stroock1975diffusion} or, in the presence of additional structure, using more specialized arguments as discussed in \autoref{sec:res:ex}.

The structure of the operator $\cA$ in the following assumption allows $Y$ to be a general Markovian semimartingale, whereas $X$ is constant. To describe the behavior of small jumps, we fix a truncation function $\chi\colon \bY\to\bY$, which is bounded, continuous, and coincides with the identity on a neighborhood of zero. 

\begin{assumption}[Operator $\cA$]\label{ass:A}
$\cD(A)=C^2_b(\bX\x\bY)$ and
\begin{align*}
\cA f(u,x,y) &= \partial_y f(x,y) \beta(u,x,y) 
+ \frac12 \partial_y^2 f(x,y) \sigma^2(u,y)
\\&\qquad
+ \int_{\bY} \Big(f(x,y+z)-f(x,y)-\partial_y f(x,y)\chi(z) 
\Big) K(u,x,y,\rd z),
\end{align*}
where $\be\colon \bU\x\bX\x\bY\to\bY$, $\sigma^2\colon \bU\x\bY\to \bY \otimes \bY$, and $K$ is a transition kernel from $\bU\x\bX\x\bY$ to $\bY\setminus\{0\}$.
\end{assumption}

The following bounds guarantee in a simple way that the value functions are finite and reduce technicalities in the proofs by avoiding additional localizations by stopping times. 

\begin{assumption}[Boundedness]\label{ass:bou}
The expressions 
\begin{align*}
b(u,x,y), 
&&
\be(u,x,y), 
&&
\sigma^2(u,y), 
&&
\int_{\bY} \big(|z|^2\wedge 1\big) K(u,x,y,\rd z)
\end{align*}
are measurable and bounded over $(u,x,y)\in\bU\x\bX\x\bY$.
\end{assumption}

The following assumption is related to Girsanov's theorem \cite[Theorem III.3.24]{Jacod} applied to the conditional laws of $Y$ given $X$. It is needed to describe the filter as a change of measure between the conditional laws.

\begin{assumption}[Girsanov]\label{ass:gir}
There exist functions $\phi_1\colon \bU\x\bY\to\bY$ and $\phi_2\colon \bU\x\bY\x\bY\to\R$ satisfying 
\begin{equation*}
\begin{aligned}
    \sigma^2(u,y) \phi_1(u,y) &= \be(u,1,y)-\be(u,0,y)
    - \int_\R \big(\phi_2(u,y,z)-1 \big) \chi(z) 
    \big( K(u,1,y,\rd z)+K(u,0,y,\rd z)\big) ,
    \\
    \phi_2(u,y,z) &=\frac{K(u,1,y,\rd z)}
    {\big(K(u,1,y,\rd z)+K(u,0,y,\rd z)\big)/2}.
\end{aligned}\end{equation*}
\end{assumption}

The following assumption on the structure of the operator $\cG$ encodes the filtering equations, which are derived in \autoref{lem:fi}. 

\begin{assumption}[Operator $\cG$]\label{ass:G}
$\cD(G)=C^2_b([0,1]\x\bY)$ and
\begin{align*}
\cG f(u,p,y) 
&= \partial_y f(p,y) \ol\beta(u,p,y) 
+\frac12 \partial_p^2 f(p,y) p^2(1-p)^2 \phi_1(u,y)^\top\sigma^2(u,y)\phi_1(u,y)
\\&\qquad
+ \partial_p\partial_y f(p,y) p(1-p)\sigma^2(u,y)\phi_1(u,y)
+ \frac12 \partial_y^2 f(p,y) \sigma^2(y,u)
\\&\qquad
+ \int_{\bY} \Big(f(p+j(u,p,y,z),y+z)-f(p,y)
\\&\qquad\qquad\quad
-\partial_p f(p,y)j(u,p,y,z)-\partial_y f(p,y)\chi(z) 
\Big) \ol K(u,p,y,\rd z),
\end{align*}
where 
\begin{align*}
\ol\beta(u,p,y) &= p\beta(u,1,y)+(1-p)\beta(u,0,y),
\\
\ol K(u,p,y,\rd z) &= p K(u,1,y,\rd z)+(1-p) K(u,0,y,\rd z),
\\
j(u,p,y,z) &= \frac{p \phi_2(u,y,z)}{p \phi_2(u,y,z)+(1-p)\big(2-\phi_2(u,y,z)\big)}-p,
\end{align*}
and where it is understood that $j(u,p,y,z)=0$ if $p\in\{0,1\}$.
\end{assumption}

The following assumption is a \citet{novikov1980conditions} condition for the uniform integrability of a stochastic exponential. It is needed in \autoref{lem:fi} to derive the filtering equations. The condition has also an information-theoretic interpretation, see \autoref{rem:asy}. The specific version of the condition is due to \citet[Th\'eor\`eme IV.3]{Lepingle}. 

\begin{assumption}[Novikov condition]\label{ass:novikov}
The following expression is bounded in $(y,u)\in\bY\x\bU$:
\begin{multline*}
    \Phi(u,y)=\frac18 \phi_1(u,y)^\top \sigma^2(u,y) \phi_1(u,y)
    \\
    + \int_{\bY} \left( 1- \sqrt{\phi_2(u,y,z)\big(2-\phi_2(u,y,z)\big)}\right) 
    \big(K(u,1,y,\rd z)+K(u,0,y,\rd z)\big) .
\end{multline*}
\end{assumption}

The following two assumptions are used to show that solutions of martingale problems related to $\cA$ and $\cG$ depend continuously on parameters (c.f. \autoref{lem:se:app}). 

\begin{assumption}[Continuity]\label{ass:con}
The expressions
\begin{align*}
\be(u,x,y), && \sigma^2(u,y), && \phi_1(u,y),
&& \int_\bY g(j(u,p,y,z),z) \ol K(u,p,y,\rd z)
\end{align*}
are continuous in $(y,u)$ for all $x\in\bX$, $p\in[0,1]$, and $g \in C_b([0,1]\x\bY)$ satisfying $g(x)=O(\lvert x\rvert^2)$ as $\lvert x\rvert \to 0$. 
\end{assumption}

\begin{assumption}[Condition on big jumps]\label{ass:bigjumps}
\begin{equation*}
\lim_{a\to\infty} \sup \Big\{
K\big(u,x,y,\{z\in\bY\colon \lvert z\rvert >a\}\big)\colon (u,x,y)\in\bU\x\bX\x\bY
\Big\}=0.
\end{equation*}
\end{assumption} 

The following two assumptions are used in various places to show that solutions of martingale problems related to the operators $\cA$ and $\cG$ exist and depend continuously on parameters. In contrast to the previous assumptions these are indirect conditions on the coefficients of the model. Some examples of how they can be verified are presented in \autoref{sec:res:ex}. General sufficient conditions are given in \cite{komatsu1973markov,stroock1975diffusion}.

\begin{assumption}[Well-posedness for the problem with partial observations]\label{ass:wel}
The martingale problem for $(\cA,F)$ is well-posed for all deterministic functions $F\colon [0,\infty)\to \bU$.
\end{assumption}

\begin{assumption}[Well-posedness for the separated problem]\label{ass:wel2}
The martingale problem for $(\cG,u)$ is well-posed\footnote{The martingale problem for $(\cG,F)$ is defined in analogy to the one for $(\cA,F)$, see \autoref{sec:app:mgp} in the Appendix.} for all $u\in \bU$.
\end{assumption}
 
\subsection{Separation and approximation result}\label{sec:equ}

\begin{theorem}[Separation and approximation]\label{thm:equiv}
The following statements hold under Assumptions \ref{ass:A}--\ref{ass:wel2}:
\begin{enumerate}[(a)]
\item The value functions of the control problems agree:
\begin{equation*}
    V(p,y):=V^\po(p,y)=V^\se(p,y)<\infty.
\end{equation*}
\item Controls can be approximated arbitrarily well in value by piecewise constant controls:
\begin{equation*}
    V(p,y)=\sup_{\de>0}V^\de(p,y),
\end{equation*} 
where $V^\de(p,y)=V^{\po,\de}(p,y)=V^{\se,\de}(p,y)$ is the value function obtained by restricting to control process $U$ which are piecewise constant on a uniform time grid of step size $\delta>0$.
\end{enumerate}
\end{theorem}
 
\begin{remark}\label{rem:after_thm_equ}
\begin{itemize}
	\item The importance of \autoref{thm:equiv} lies in its capacity to transform the control problem with partial observations into a problem which can be analyzed and solved by standard methods like dynamic programming or linear programming (see \autoref{sec:lit:sep} for some background and further references). The approximation result guarantees that the class of admissible strategies is small enough to exclude degeneracies like \autoref{ex:pat}. It is also large enough to guarantee the existence of optimal strategies in the restless bandit problem presented in \autoref{sec:res}. In the general case, existence of optimal strategies can be guaranteed by the standard technique of allowing relaxed (measure-valued) controls, as described in \autoref{rem:se:top}. 
\item The intuition behind \autoref{thm:equiv} is that rational Bayesian agents base their strategy on the posterior distribution $P_t$ of the unknown state and the public information $Y_t$. 
\end{itemize}
\end{remark}
    
\autoref{thm:equiv} follows from a sequence of lemmas, which can be found in \autoref{sec:app:sto} in the Appendix. We now give a verbal proof of the theorem, highlighting the role that each individual lemma plays.

\begin{proof}[Proof of \autoref{thm:equiv}]
By \autoref{ass:bou} the reward function $b$ is bounded, which implies that all value functions are finite. If $(U,X,Y)$ is a control with partial observations and $P$ is a c\`adl\`ag version of the martingale $\bE[X\mid\cF^Y_t]$, then $(U,P,Y)$ is a separated control with the same value by \autoref{lem:fi}. Taking the supremum over all controls or step controls, one obtains that
\begin{align*}
    V^{\po}(p,y)\leq V^{\se}(p,y), 
    &&
    V^{\po,\de}(p,y)\leq V^{\se,\de}(p,y).
\end{align*}
By \autoref{lem:se:app}, separated controls can be approximated arbitrarily well in value by separated step controls. Formally, this is expressed by the equation
\begin{equation*}
    \sup_{\de>0} V^{\se,\de}(p,y)= V^\se(p,y).
\end{equation*}
In \autoref{lem:se_to_po}, it is shown that Markovian step controls of the separated problem can be transformed into controls of the problem with partial observations of the same value. This is done by a recursive construction, stitching together solutions $(X,Y)$ of the martingale problem associated to $\cA$ under constant controls corresponding to each step of the control process. As optimal Markovian controls exist for the discretized separated problem,
\begin{equation*}
    V^{\se,\de}(p,y) \leq V^{\po,\de}(p,y).
\end{equation*}
Taken together, this implies that 
\begin{equation*}
V^{\se,\de}(p,y) = V^{\po,\de}(p,y)
\end{equation*}
and
\begin{equation*}
    V^\po(p,y) \leq V^{\se}(p,y)
    =\sup_\de V^{\se,\de}(p,y)
    = \sup_\de V^{\po,\de}(p,y) 
    \leq V^\po(p,y).  \qedhere
\end{equation*}
\end{proof}

\section{A restless bandit model.}\label{sec:res}

We introduce and solve a specific restless bandit model motivated by dynamic complementarities in the production of human capital.\footnote{We point out that the use of our model is not restricted to human capital production. Complementarities between past and future investments arise in many other applications such as e.g.~ job search (see \autoref{sec:introduction}).} The bandit model has a ``safe'' arm with constant payoffs corresponding to the absence of investment in human capital. The second arm is ``risky'' and corresponds to investment in human capital. The risky arm has stochastic payoffs, which depend on an unobserved ``type'' $X$ of the agent and her level of ``human capital'' $H$. The key assumption of the model is that investments increase the level of human capital, which in turn renders future investments more profitable (\autoref{ass:mon}). This complementarity is well documented in the literature on human capital formation (see e.g. \citet{CunhaHeckman2010} and references therein). Mathematically speaking, it represents the only departure from the class of L\'evy bandits studied by \citet{CohenSolan}. 

The restless bandit model is formulated in \autoref{sec:res:set}. Some examples are given in \autoref{sec:res:ex}, and the model is solved in \autoref{sec:red}. The asymptotics of the filter and strategy turn out to be similar to the classical case (see Sections \ref{sec:res:ind}--\ref{sec:res:tra}), but an important and potentially empirically relevant difference emerges in the analysis of populations of agents in \autoref{sec:res:evo}: in the long-run, all agents move away from the decision frontier. This makes untargeted incentives for investment ineffective and is one of the main motivations for the model at hand. 

\subsection{Setup and assumptions.}\label{sec:res:set}

The general framework of \autoref{sec:sto}, including Assumptions \ref{ass:A}--\ref{ass:wel2}, remains in place. The following structural assumption encodes that 
\begin{inparaenum}[(a)]
\item the observation process $Y=(H,R)$ takes values in $\bH\x\R=\R^2$,
\item the process $H$ has deterministic increments depending only on $U$ and $H$, 
\item under a choice $U=0$ of the safe arm, the reward process $R$ has constant increments, and
\item under a choice $U=1$ of the risky arm, the reward process $R$ has stochastic increments depending on $X$ and $H$.
\end{inparaenum}

\begin{assumption}[Structural assumption]\label{ass:str}
$\bU=\{0,1\}$, $\bY=\bH\x\R=\R^2$, $Y=(H,R)$. The coefficients $(\beta,\sigma,K)$ of the generator $\cA$ in \autoref{ass:A} are of the form 
\begin{equation*}
\begin{aligned}
\beta(1,x,h,r)&=
\begin{pmatrix}\beta_H(1,h)\\ \beta_R(x,h)\end{pmatrix},
&
\beta(0,x,h,r)&=
\begin{pmatrix}\beta_H(0,h)\\ k\end{pmatrix},
\\
\sigma^2(1,h,r)&=
\begin{pmatrix}
0 & 0 \\ 0 & \sigma_R^2(h)
\end{pmatrix},
&
\sigma^2(0,h,r)&=
\begin{pmatrix}
0 & 0 \\ 0 & 0
\end{pmatrix},
\\
K(1,x,h,r,\rd h,\rd r)&=
\delta_0(\rd h) K_R(x,h,\rd r),
&
K(0,x,h,r,\rd h,\rd r) &= 0,
\end{aligned}\end{equation*}
where
\begin{align*}
k\in\R,
&&
\beta_H\colon\bU\x\bH\to\bH,
&&
\beta_R\colon\bX\x\bH\to\R,
&&
\sigma_R:\bH\to\R,
\end{align*}
and $K_R$ is a transition kernel from $\bX\x\bH$ to $\R \setminus \{0\}$ satisfying $\sup_{x,h}\int_\R |r|^2\wedge |r| K_R(x,h,\rd r)<\infty$.
\end{assumption}

In line with the literature on L\'evy bandits, the reward received at time $t$ is the infinitesimal increment $\rd R_t$. To rewrite this in terms of a reward function $b(U_t,X,H_t)$ we impose the condition
\begin{equation*}
\bE\left[\int_0^\infty \rho e^{-\rho t} \rd R_t\right] 
= 
\bE\left[\int_0^\infty \rho e^{-\rho t} b(U_t,X,H_t) \rd t\right]. 
\end{equation*}
\autoref{lem:r} shows that this condition leads to the following specification of the reward function $b$.

\begin{assumption}[Reward function]\label{ass:b}
The reward function $b\colon \bU\x\bX\x\bH$ is given by
\begin{equation*}
b(u,x,h) = 
\left\{
\begin{aligned}
&\be_R(x,h)+\int_\R \big(r-\chi(r)\big) K_R(x,h,\rd r), 
&\text{if $u=1$,}
\\
&k,
&\text{if $u=0$.}
\end{aligned}
\right.
\end{equation*}
\end{assumption}

By the following assumption, investment in the risky arm makes future investments in the risky arm more profitable. This dependence is mediated by the process $H$, which increases with investment in the risky arm and decreases otherwise. 

\begin{assumption}[Monotonicity condition]\label{ass:mon}
The condition $\be_H(0,h)\leq0\leq\be_H(1,h)$ holds for all $h\in\bH$. Moreover, the reward $b(1,x,h)$ of the risky arm is non-decreasing in $x\in\bX$ and $h\in\bH$.
\end{assumption}

\subsection{Examples.}\label{sec:res:ex}

We show how some well-known classical bandit models described in \autoref{sec:lit:ban} can be extended to restless bandit models, which naturally fit into the framework of this paper and satisfy Assumptions~\ref{ass:A}--\ref{ass:mon}. A common feature of our extension is the presence of an auxiliary state variable $H_t$, whose dynamics are given by the ODE 
\begin{align*}
\rd H_t&=\beta_H(U_t,H_t)\rd t, 
&
H_0&=0,
\end{align*}
for some function $\beta_H:\bU\times\bH\to\bH$ such that the ODE is well-posed under any deterministic control $U:[0,\infty)\to\mathbb U$. For example, this is the case if $H$ increases or decreases linearly depending on the strategy:
\begin{align*}
	H_0&=0, & \beta_H(0,h)&=-1, & \beta_H(1,h) &= 1.
\end{align*} 
The purpose of the auxiliary state variable $H_t$ is to make the risky arm more or less profitable depending on the amount of past investment in the risky arm. We show in Examples~\ref{ex:res:ex:gau}--\ref{ex:res:ex:lev} below how this can be done for Gaussian, Poisson, and L\'evy bandits.

\begin{example}[Gaussian bandits]\label{ex:res:ex:gau}
In the Gaussian bandit model introduced by \citet{Karatzas} the reward of the risky arm is a diffusion whose drift depends on the unobservable type $X$. This model becomes restless if the drift depends  additionally on the level of human capital:
\begin{align*}
	dR_t &= U_t \beta_R(X,H_t) dt  + U_t \sigma_R(H_t) dW_t + (1-U_t)k.
\end{align*}
Then $(X,Y)=(X,H,R)$ is a controlled Markov process, and its generator $\mathcal A$ has the structure described in \autoref{ass:A,ass:str}. \autoref{ass:wel} holds automatically thanks to \autoref{ass:str} and the well-posedness of the ODE for $H_t$ under deterministic controls.\footnote{This follows from \cite[Theorem~III.2.16]{Jacod} noting that $R$ has deterministic semimartingale characteristics under any deterministic control. We refer to \cite{Jacod, Protter} and references therein for more general conditions for the well-posedness of stochastic differential equations and martingale problems.} If $\beta_R$ and $\sigma_R$ are bounded continuous functions, $\beta_R(0,h)\leq \beta_R(1,h)$, and $\sigma_R(h)>0$, then Assumptions~\ref{ass:A}--\ref{ass:mon} are satisfied.\footnote{\autoref{ass:wel2} is satisfied because the coefficients of $\mathcal G$ are bounded and Lipschitz continuous (see \cite[Theorem~III.2.32]{Jacod} or \cite{komatsu1973markov,stroock1975diffusion}). The verification of all other assumptions is straightforward.} 
\end{example}

\begin{example}[Poisson bandits]\label{ex:res:ex:poi}
In Poisson bandits, which were introduced by \citet{Presman}, the reward of the risky arm is a Poisson process $N$ whose jump intensity depends on the unobservable type $X$. As an extension we allow the jump intensity to depend additionally on the current level of human capital $H_t$. Then the jump intensity becomes a function $\lambda(X,H_t)$ of $X$ and $H_t$, and we set 
\begin{align*}
	dR_t &= U_t dN_t + (1-U_t) k, 
	&
	dN^p_t= \lambda(X,H_t)dt,
\end{align*}
where $N^p$ denotes the compensator\footnote{See \cite[Theorem~3.17]{Jacod} for the definition of compensator or dual predictable projection.} of the Poisson process $N$. Equivalently, the compensator of the jump measure of $R$ is $K_R(X,H_t,dr)dt$, where $K_R(x,h,dr)=\lambda(x,h)\delta_1(dr)$. If $\lambda$ is a continuous bounded function satisfying $0\leq\lambda(0,h)\leq \lambda(1,h)$, then Assumptions~\ref{ass:A}--\ref{ass:mon} hold by the same reasoning as above.\footnote{\autoref{ass:wel2} is satisfied because the martingale problem for $\mathcal G$ is piecewise deterministic with finitely many jumps at exponential stopping times.}
\end{example}

\begin{example}[L\'evy bandits]\label{ex:res:ex:lev}
L\'evy bandits, which were introduced by \citet{CohenSolan}, generalize the class of Gaussian and Poisson bandits. They are characterized by the L\'evy triplet of the reward process, which depends on the unobserved type $X$. In our extension to a restless bandit model it may depend additionally on the current level of human capital. The characterization of the reward process in terms of L\'evy triplets is equivalent to the formulation in terms of the martingale problem for $\mathcal A$. A sufficient condition for \autoref{ass:wel2} is that the jump measures $K_R(1,h,dr)$ and $K_R(0,h,dr)$ are equivalent for each $h$.\footnote{\autoref{ass:wel2} follows from \cite[Theorem~3.3]{Ceci2002}, noting that uniqueness holds for the filtered martingale problem for $\mathcal A$ as shown in Step 1 of the proof of \autoref{lem:se:app}.} \autoref{ass:wel} holds by the reasoning above, and all other assumptions are direct conditions on the model coefficients. 
\end{example}

All three examples are genuinely restless bandit models because the reward structure of the risky arm decreases while the risky arm is inactive. Optimal strategies for these models are provided by \autoref{thm:main}. Some important differences to classical bandit models are pointed out in \autoref{sec:res:evo}.

\subsection{Reduction to optimal stopping.}\label{sec:red}

\begin{definition}[Gittins' index]\label{def:git}
Gittins' index $G$ is defined by\footnote{The index does not depend on the initial value of $R$ (see \autoref{lem:no_r}). The two expressions for $G$ in \autoref{def:git} are shown to be equivalent in \citet{ElKaroui1994}.}
\begin{equation*}
    G(p,h)=\inf\left\{s: \sup_T \bE\left(\int_0^T \rho e^{-\rho t} 
    (\rd R_t-s\rd t)\right)
    \leq 0\right\}
    =\sup_T \frac{\bE\left(\int_0^T \rho e^{-\rho t} \rd R_t\right)}
    {\bE\left(\int_0^T \rho e^{-\rho t}\rd t\right)},
\end{equation*}
where $(1,P,H,R)$ is a separated control with constant control process $U\equiv 1$ and initial condition $(P_0,H_0)=(p,h)$, and where the suprema are taken over all $\{\cF^{P,H}_t\}$-stopping times $T$.
\end{definition}

\begin{theorem}[Optimal stopping]\label{thm:main}
The following statements hold under Assumptions \ref{ass:A}--\ref{ass:mon}.
\begin{enumerate}[(a)]
\item The value function $V$ (see \autoref{thm:equiv}) does not depend on the initial value of the process $R$ and can be written as $V=V(p,h)$. 
\item The strategy $U^*_t=\one_{\llbracket 0,T^* \rrbracket}(t)$ is optimal, where
\begin{equation*}
T^*=\inf \{t\geq 0:V(P_t,H_t)\leq k\}=\inf \{t\geq 0: G(P_t,H_t)\leq k\}.
\end{equation*}
\end{enumerate}
\end{theorem}

\begin{remark}
\begin{itemize}
\item The main value of \autoref{thm:main} is that it reduces the restless bandit problem to an optimal stopping problem. This exhibits the structure of optimal strategies in terms of a decision frontier (see \autoref{prop:frontier}). Moreover, the stopping problem can be solved more easily by a variety of specialized methods (see e.g. \citet[Chapter~IV]{Peskir}). 
\item The intuition behind \autoref{thm:main} is that choosing the risky arm early rather than late has two advantages: first, it reveals useful information about the hidden state $X$ early on, and second, it makes future rewards from the risky arm more profitable without depreciating rewards from the safe arm.

\item The elimination of the state variable $r$ is possible because of \autoref{ass:str}, which asserts that the drift, volatility, and jump measure of the reward process only depend on $P$ and $H$. 

\item At the heart of \autoref{thm:main} lies the assertion that any optimal control of the discretized problem can be transformed into a stopping rule of at least the same value (\autoref{lem:os}). The argument is based on \citet[Theorem 5.2.2]{Berry}, but in our setting rewards may depend on the history of experimentation with the risky arm. This dependence is subject to the monotonicity properties in \autoref{ass:mon}. Our proof shows that these properties are exactly what is needed to adapt the argument of \citet{Berry} to a restless bandit setting. 

\item The strategy $U^*$ is well-defined and optimal for the separated problem as well as the problem with partial observations. 
\end{itemize}

\end{remark}

\autoref{thm:main} follows from a sequence of lemmas, which can be found in \autoref{sec:app:res} in the Appendix. The following proof explains the r\^ole that each individual lemma plays. 

\begin{proof}[Proof of \autoref{thm:main}]
The value function does not depend on the initial value of $R$ by \autoref{lem:no_r}. Therefore, it can be written as $V(p,h)$. The discrete-time value function $V^\delta(p,h)$ is non-decreasing in $(p,h)$ and convex in $p$. This is established in \autoref{lem:mon} using the monotonicity properties in \autoref{ass:mon}. The result is used in \autoref{lem:ini} to prove a sufficient condition for the optimality of the risky arm in the discretized problem: if the myopic payoff is higher for the risky than for the safe arm, then choosing the risky arm is uniquely optimal. This sufficient condition is used in \autoref{lem:os} to prove that $V^\de(p,h)$ is a supremum of values of stopping rules. The approximation result of \autoref{thm:equiv} implies that $V(p,y)$ is also a supremum of values of stopping rules. The  stopping time $T^*=\inf\{t\geq 0\colon V(P_t,H_t)\leq k\}$ is optimal by \autoref{lem:os:exi}. The alternative characterization of $T^*$ in terms of Gittins' index is well-known, see e.g. \citet[Theorem 2.1]{Morimoto} or \citet[Proposition 3.4]{ElKaroui1994}. 
\end{proof}

An immediate consequence of \autoref{thm:main} is a characterization of optimal strategies by a curve which is typically called the \emph{decision frontier}.

\begin{proposition}[Decision frontier]\label{prop:frontier}
There is a curve in the $(p,h)$-domain
such that it is optimal to invest in the risky arm if $(P_t,H_t)$ lies to the right and above of the curve. Otherwise, it is optimal to invest in the safe arm. 
\end{proposition}

\begin{proof}
The value function $V(p,h)$ is non-decreasing in its arguments by \autoref{lem:mon} and bounded from below by the constant $k$. The desired curve is the boundary of the domain $\{(p,h):V(p,h)>k\}$. The characterization of optimal strategies via the position of $(P_t,H_t)$ relative to the curve follows from \autoref{thm:main}. 
\end{proof}

\subsection{Indexability.}\label{sec:res:ind}

Another consequence of \autoref{thm:main} is the indexability of our restless bandit model in the sense of \citet{Whittle1988}. 

\begin{definition}[Indexability]
Consider a two-armed bandit problem with a safe and a risky arm. The bandit problem is called indexable if the set of states where the safe arm is optimal is increasing in the payoff $k$ of the safe arm. 
\end{definition}
	
\begin{proposition}[Indexability]\label{prop:indexability}
The restless bandit model of \autoref{sec:res:set} is indexable.
\end{proposition}

\begin{proof}
Gittins' index $G(p,h)$ depends only on the payoff of the risky arm. Therefore, the set $\{(p,h)\colon G(p,h)\leq k\}$ where the safe arm is optimal has the required properties. 
\end{proof}

\subsection{Asymptotic learning.}\label{sec:res:asy}

\begin{definition}[Asymptotic learning and infinite investment]
For any $\om \in \Om$, we say that asymptotic learning holds if $\lim_{t\to\infty} P_t(\omega) = X(\omega)$. We say that the agent invests an infinite amount of time in the risky arm if $\int_0^\infty U_t(\omega) \rd t = \infty$. 
\end{definition}

\begin{assumption}[Bounds on the flow of information]\label{ass:asy}
The initial belief is non-doctrinaire, i.e., $P_0\in(0,1)$. The measures $K_R(1,h,\cdot)$ and $K_R(0,h,\cdot)$ are equivalent, for all $h \in \bH$. The function $\Phi(1,\cdot)$ defined in \autoref{ass:bou} is bounded from below by a positive constant. 
\end{assumption}

\begin{proposition}[Asymptotic learning]\label{prop:convergence} 
Under Assumptions \ref{ass:A}--\ref{ass:asy}, the following statements hold:
\begin{enumerate}[(a)]
\item Under any control, asymptotic learning occurs if and only if the agent invests an infinite amount of time in the risky arm.
\item Under the optimal control of \autoref{thm:main}, asymptotic learning takes place if and only if $(P,H)$ remains above the decision frontier for all time. 
\end{enumerate}
\end{proposition}

\begin{proof}
(a) follows from \autoref{lem:asy}. (b) follows from (a) and the characterization of optimal controls in \autoref{prop:frontier}. 
\end{proof}

\begin{remark}\label{rem:asy}
\begin{itemize}
\item The limit $\lim_{t\to\infty} P_t$ exists almost surely because $P$ is a bounded martingale. If the belief $P_0\in\{0,1\}$ is doctrinaire, then the belief process $P$ is constant and equal to the hidden state $X$.

\item Agents can learn their true type $X$ in two ways: either through a jump of the belief process to $X$, or through convergence to $X$ without a jump to the limit. The first kind of learning is excluded by the equivalence of $K_R(1,h,\cdot)$ and $K_R(0,h,\cdot)$. The second kind of learning is characterized by divergence of the Hellinger process of the measures $\bP_1$ and $\bP_0$. The Hellinger process is closely related to the function $\Phi(u,y)$, which can be interpreted as the informativeness of the arm $u$ about the state $X$. The upper and lower bounds on $\Phi$ in \autoref{ass:novikov,ass:asy} establish an equivalence between divergence of the Hellinger process and divergence of the accumulated amount of investment in the risky arm (see \autoref{lem:asy}).

\item If the measures $K_R(1,H,\cdot)$ and $K_R(0,H,\cdot)$ are not equivalent, the belief process $P$ jumps to the true state $X$ with positive probability on any finite interval of time where the risky arm is chosen. For example, this is the case in the exponential bandits model of \citet{KellerRadyCripps}.

\item \autoref{prop:convergence} can be contrasted with the strategic experimentation model of \citet{BoltonHarris2000} and the social learning model of \citet[Example 1.1]{Acemoglu}. In these models, asymptotic learning always takes place because agents continuously receive information about the hidden state, regardless of whether they choose to invest or not. 
\end{itemize}	
\end{remark}

\subsection{Comparison to the full-information case.}\label{sec:res:tra}

By the full-information case, we mean the bandit model where the otherwise hidden state variable $X$ is fully observable. This model is equivalent to the model with partial observations and $P_0\in\{0,1\}$. It follows from \autoref{thm:main} and the monotonicity condition in \autoref{ass:mon} that the optimal strategy in the full-information case is constant in time and given by $\one_{V(X,H_0)>k}$. 

\begin{definition}[Asymptotic efficiency]\label{def:eff}
For any $t \geq 0$ and $\om \in \Om$, $U_t(\omega)$ is called efficient if it coincides with $\one_{V(X(\omega),H_0)>k}$. Moreover, $U(\omega)$ is called asymptotically efficient if $U_t(\omega)$ is efficient for all sufficiently large times $t$.
\end{definition}

\begin{assumption}[Decision frontier stays away from $p=0$ and $p=1$]\label{ass:eff}
There is $\epsilon>0$ such that for all $h\in\bH$, $V(\epsilon,h)=k$ and $V(1,h)>k$.	
\end{assumption}

\begin{proposition}[Asymptotic efficiency]\label{prop:eff}
Let Assumptions \ref{ass:A}--\ref{ass:eff} hold, let $U$ be the optimal strategy provided by \autoref{thm:main}, and assume that $(P_0,H_0)$ lies above the decision frontier. Conditional on $X=0$, asymptotic efficiency holds almost surely. Conditional on $X=1$, however, asymptotic efficiency may hold and fail with positive probability. 
\end{proposition}

\begin{proof}
If $X=0$, investment in the risky arm can't continue forever. Otherwise, $P_t$ would converge to zero by \autoref{prop:convergence}. As the decision frontier is strictly bounded away from the set $p=0$, $(P_t,H_t)$ would eventually drop below the decision frontier, a contradiction. Thus, investment stops at some finite point in time. This is efficient given $X=0$ because $V(0,H_0)=k$.

If $X=1$, then $(P,H)$ may or may not drop below the frontier at some point in time. Both cases may happen with positive probability. In the former case, the agent stops investing, which is inefficient because $V(1,H_0)>0$. In the latter case, the agent never stops investing, which is efficient. 
\end{proof}

\begin{remark}\label{rem:eff}
\begin{itemize}
\item Efficiency holds if there is some time $t$ where the agent's plan for future investments is the same as if she had known $X$ from the beginning. Of course, this still leaves open the possibility that some early investment decisions were inefficient.

\item The intuition behind \autoref{prop:eff} is that a sequence of bad payoffs can lead agents to refrain from experimentation with the risky arm. For agents of the type $X=0$, this is efficient, but for agents with $X=1$, it is not. In this regard, the restless bandit model behaves as a standard bandit model.

\item It follows that in the long run, compared to a setting with full information, agents invest too little in the risky arm. This points to the importance of policies designed to increase investment in the risky arm. 

\item \autoref{ass:eff} limits the influence of $H$ on the rewards from the risky arm: the safe arm is optimal if $X=0$ is known for sure, regardless of how high $H$ is, and similarly the risky arm is optimal if $X=1$ is known for sure, regardless of how low $H$ is.

\item Without \autoref{ass:eff}, it is still possible to characterize asymptotic efficiency using the necessary and sufficient conditions of \autoref{lem:asy}, but there are more cases to distinguish. Some of them have no counterpart in classical bandit models. For example, there can be low-type agents who invest in the risky arm at all times. This can be either efficient or inefficient, depending on whether $V(0,H_0)$ exceeds $k$. Similarly, it can be efficient or inefficient for high-type agents to stop investing, depending on $V(1,H_0)$.
\end{itemize}	
\end{remark}

\subsection{Evolution of a population of agents.}\label{sec:res:evo}

Assume that there is a population of agents with initial states $(P_0,H_0)$, which might vary from agent to agent. Moreover, assume that agents have independent types, such that learning from others is impossible. Alternatively, learning could be precluded by making actions and rewards private information. Then all agents behave as in the single player case. The distribution of agents in the $(p,h)$-domain evolves over time and converges to the distribution of $(P_\infty,H_\infty)$. 

\begin{proposition}\label{prop:pop}
Let Assumptions \ref{ass:A}--\ref{ass:eff} hold, let $p^*(h)$ denote the decision frontier, and consider a population of agents with $(P_0,H_0)$ above the decision frontier. 
\begin{enumerate}[(a)]
\item In a restless bandit model with $\beta(0,h)<0<\beta(1,h)$, $(P_\infty,H_\infty)$ satisfies
\begin{align*}
&P_\infty \in [0, p^*(-\infty)] \text{ and } H_\infty=-\infty 
&&\text{or}&& P_\infty=1 \text{ and } H_\infty=\infty.
\end{align*}

\item In a classical bandit model with $\beta(0,h)=\beta(1,h)=0$ and $\Delta P \geq -\epsilon$ for some $\epsilon\geq 0$, 
\begin{equation*}
P_\infty \in 
[p^*(H_0)-\epsilon,p^*(H_0)] \cup \{1\}
\text{ and } H_\infty = H_0.
\end{equation*}	
In particular, agents in models without jumps either end up right at the decision frontier $(p^*(H_0),H_0)$ in finite time or converge to $(1,H_0)$.
\end{enumerate}
\end{proposition}

\begin{proof}
\begin{inparaenum}[(a)]
\item  If $(P,H)$ drops below the decision frontier, $P$ is frozen and $H$ decreases to $-\infty$. Otherwise, $P$ increases to $1$ and $H$ to $\infty$. 
\item $H$ is constant and $P$ either converges to $1$ or drops below the decision frontier and remains there forever.
\end{inparaenum} 
\end{proof}

\begin{remark}\label{rem:res:evo}
\begin{itemize}
\item \autoref{prop:pop} shows that agents in classical bandit models accumulate at or near the decision frontier, whereas they drift away from the frontier in restless bandit models. This leads to different predictions about the effectiveness of incentive schemes designed to increase investment in the risky arm. 

\item To wit, consider a subsidy for investment in the risky arm or, alternatively, a penalty for investment in the safe arm. These incentives lower the decision frontier in the $(p,h)$-domain. Some agents, who were previously below the frontier, will now find themselves above the frontier and will find it optimal to start investing in the risky arm again. The number of such agents can be expected to be very small in restless bandit models because agents keep drifting away from the frontier once they stopped investing in the risky arm. Consequently, incentives have negligible effects on average investment, in particular if they are carried out late in time. In contrast, in classical bandit models even small shifts of the decision frontier have large effects on average investment because there are many agents at or near the frontier, namely all agents who ever stopped investing in the risky arm. 

\item Thus, our model provides an explanation for the ineffectiveness of subsidies designed to boost investment in projects with uncertain payoffs. Our explanation does not rely on switching costs. 
\end{itemize}	
\end{remark}

\section{Conclusions.}\label{sec:con}

We presented an extension of classical bandit models of investment under uncertainty motivated by dynamic aspects of resource development. The extension is new and has economic significance in a wide range of real world settings. 

We dealt with the delicate issue of setting up the control problem with partial observations in continuous time. As explained in \autoref{sec:lit:sep}, recent standard formulations of optimal control under partial observation do not apply in our general setting. In addition to its importance to the theory of optimal control, our solution is also a contribution to the bandit literature. 

Our framework encompasses both the exponential bandit model of \citet{KellerRadyCripps}, where jumps can occur only for high type agents, and the Poisson and Levy bandit models of \citet{KellerRady2010, KellerRady2015} and \citet{CohenSolan}, where it is assumed that one jump measure is absolutely continuous with respect to the other.

We solved the restless bandit model by an unconventional approach.  Instead of using the HJB equation or a setup using time changes, we discretized the problem in time and showed that any optimal strategy can be modified such that the agent never invests after a period of not investing and such that the modified strategy is still optimal. 

Our models constitute a new class of indexable restless bandit models. While other classes of indexable bandits are known, they either involve no learning about one's type (\citet{Glazebrook2006b}), do not allow history-dependent payoffs (\citet{Washburn}), or are restricted to very specific reward processes (e.g. finite-state Markov chains as in \citet{NinoMora2001}). 

\appendix

\section{Notation.}\label{sec:app:not}

For any Polish space $\bS$, $B(\bS)$ will denote the space of $\R$-valued Borel-measurable functions on $\bS$, $C(\bS)$ the continuous functions, $C_b(\bS)$ the bounded continuous functions, and $\mathcal P(\bS)$ the space of probability measures on $\bS$. $D_{\bS}[0,\infty)$ denotes the space of $\bS$-valued c\`adl\`ag functions on $[0,\infty)$ with the Skorokhod topology, $L_{\bS}[0,\infty)$ the c\`agl\`ad functions, and $C_{\bS}[0, \infty)$ the subspace of continuous functions. If $\bS$ is endowed with a differentiable structure, then $C^k_b(\bS)$ denotes the functions with $k$ bounded continuous derivatives. 

Throughout the paper, all filtrations are assumed to be complete, and all processes are assumed to be progressively measurable. The law of a random variable $X$ is denoted by $\mathcal L(X)$. The completion of the filtration generated by a process $Y$ is denoted by $\{\cF^Y_t\}$. If $Y$ has left-limits, they are denoted by $Y_-$, i.e., $Y_{t-}=\lim_{s\nearrow t}Y_t$. If $Y$ is of finite variation, $\on{Var}(Y)$ denotes its variation process. $H \bullet Y$ denotes stochastic integration of a predictable process $H$ with respect to a semimartingale $Y$ and $H * \mu$ with respect to a random measure $\mu$. $I$ denotes the identity process $I_t=t$. When $T$ is a stopping time, we write $Y^T$ and $\mu^T$ for the stopped versions of $Y$ and $\mu$. Stochastic intervals are denoted by double brackets, e.g., $\llbracket 0,T\rrbracket \subset [0,\infty]\x\Omega$. $Y^c$ denotes the continuous local martingale part of $Y$. A superscript $\top$ denotes the transpose of a matrix or vector. 

\section{Controlled martingale problems.}\label{sec:app:mgp}

\begin{definition}[Martingale problem for $(\cA,F)$]\label{def:app:mgp_A}
Let $F$ be a c\`agl\`ad adapted $\bU$-valued process on the space $D_{\bY}[0,\infty)$ with its canonical filtration.
\begin{enumerate}[(i)]
\item $(X,Y,T)$ is a solution of the stopped martingale problem for $(\cA,F)$ if there exists a filtration $\{\cF_t\}$, such that $X$ is an $\cF_0$-measurable $\bX$-valued random variable, $Y$ is an $\{\cF_t\}$-adapted c\`adl\`ag $\bY$-valued process, $T$ is an $\{\cF_t\}$-stopping time, and for each $f \in \cD(\cA)$, 
\begin{equation}
f(X,Y_{t\wedge T})-f(X,Y_0)
-\int_0^{t\wedge T}\cA f(F(Y)_s,X,Y_s) \rd s
\end{equation}
is an $\{\cF_t\}$-martingale. 

\item If $T=\infty$ almost surely, then $(X,Y)$ is a solution of the martingale problem for $(\cA,F)$.

\item $(X,Y)$ is a solution of the local martingale problem for $(\cA,F)$ if there exists a filtration $\{\cF_t\}$ and a sequence of $\{\cF_t\}$-stopping times $\{T_n\}$ such that $T_n\to\infty$ almost surely and for each $n$, $(X,Y,T_n)$ is a solution of the stopped martingale problem for $(\cA,F)$.

\item Local uniqueness holds for the martingale problem for $(\cA,F)$ if for any solutions $(X',Y',T')$, $(X'',Y'',T'')$ of the stopped martingale problem for $(\cA,F)$, equality of the law of $(X',Y'_0)$ and $(X'',Y''_0)$ implies the existence of a solution $(X,Y,S'\vee S'')$ of the stopped martingale problem for $(\cA,F)$ such that $(X_{\cdot\wedge S'},S')$ has the same distribution as $(X'_{\cdot\wedge T'},T')$, and $(X_{\cdot\wedge S''},S'')$ has the same distribution as $(X''_{\cdot\wedge T''},T'')$.

\item The martingale problem for $(\cA,F)$ is well-posed if local uniqueness holds for the martingale problem for $(\cA,F)$ and for each $\nu \in \cP(\bX\x\bY)$, there exists a solution $(X,Y)$ of the local martingale problem for $(\cA,F)$ such that the law of $(X,Y_0)$ is $\nu$.
\end{enumerate}
\end{definition}

\begin{definition}[Martingale problem for $(\cG,F)$]
Let $F$ be a c\`agl\`ad adapted $\bU$-valued process on $D_{[0,1]\x\bY}[0,\infty)$ with its canonical filtration.
\begin{enumerate}[(i)]
\item $(P,Y,T)$ is a solution of the stopped martingale problem for $(\cG,F)$ if there exists a filtration $\{\cF_t\}$, such that $(P,Y)$ is an $\{\cF_t\}$-adapted c\`adl\`ag $[0,1]\x\bY$-valued process, $T$ is an $\{\cF_t\}$-stopping time, and for each $f \in \cD(\cG)$, 
\begin{equation}
f(P_{t\wedge T},Y_{t\wedge T})-f(P_0,Y_0)
-\int_0^{t\wedge T}\cA f(F(P,Y)_s,P_s,Y_s) \rd s
\end{equation}
is an $\{\cF_t\}$-martingale. 

\item Solutions of the (local) martingale problem, local uniqueness, and well-posedness are defined in analogy to \autoref{def:app:mgp_A}.
\end{enumerate}
\end{definition}

\section{Noncommutativity of filtering and relaxation.}\label{sec:app:non}

To see the non-commutativity between filtering and relaxation, let us tentatively define relaxed controls with partial observations as tuples $(\Lambda,X,Y)$ such that for each $f \in \cD(\cA)$,
\begin{equation}
f(X,Y_t)-f(X_0,Y_0)-\int_0^t\int_{\bU} \cA f(u,X,Y_{s-})\Lambda_s(\rd u)\rd s
\end{equation}
is a martingale, where $\Lambda$ is a $\{\cF^Y_t\}$-predictable $\cP(\bU)$-valued process. If a well-posedness condition similar to the one in \autoref{def:po:co} holds and $P_t=\bE[X\mid\cF^Y_t]$ is the filter, then it can be shown\footnote{This follows by adapting the proof of \autoref{lem:fi} to relaxed control processes.} that a jump $\Delta Y_t$ of the observable process leads to a jump $\Delta P_t=\ol j(\Lambda_t,P_{t-},Y_{t-},\Delta Y_t)$ of the filter, where
\begin{equation}
\ol j(\lambda,p,y,z) = \frac{\int_\bU p \phi_2(u,y,z)\lambda(\rd u)}{\int_\bU \Big(p  \phi_2(u,y,z)+(1-p)\big(2-\phi_2(u,y,z)\big)\Big)\lambda(\rd u)}-p.
\end{equation}
Thus, $\Delta P_t$ is uniquely determined by $\Delta Y_t$ and the information before $t$. In contrast, this is not the case in the relaxation of the separated control problem, where a jump $\Delta Y_t$ can lead to different values of $\Delta P_t$. Indeed, the jump measure of $(P,Y)$ is compensated by the predictable random measure
\begin{equation}
\nu(\rd p,\rd y) = 
\int_\bY \delta_{j(u,P_{t-},Y_{t-},y)}(\rd p)
\ol K(u,P_{t-},Y_{t-},\rd y)\Lambda_t(\rd u).
\end{equation}

An interpretation is that the two cases differ in how uncertainty regarding $u$ is handled. In the former case, the control $u$ in the support of $\Lambda_t$ is treated as unknown in the process of updating the filter. Therefore, the jump height of the filter depends on $\Lambda_t$, but not on a random choice of $u$ in the support of $\Lambda_t$. In the latter case, however, $u$ is treated as known but random. Different choices of $u$ in the support of $\Lambda_t$ might lead to different probabilities for a jump $\Delta Y_t$, and consequently to different jumps of the filter. 

\section{Proofs of \autoref{sec:sto}.}\label{sec:app:sto}

Lemmas \ref{lem:fi}--\ref{lem:se_to_po} below are used to establish \autoref{thm:equiv}. Assumptions \ref{ass:A}--\ref{ass:wel2} are in place.


\begin{lemma}[Filtering]\label{lem:fi}
If $(U,X,Y)$ is a control with partial observations and $P$ is a c\`adl\`ag version of the martingale $\bE[X\mid\cF^Y_t]$, then $(U,P,Y)$ is a separated control of the same value as $(U,X,Y)$. 
\end{lemma}

\begin{proof}
\emph{Step 1 (Filter as change of measure from $\bP$ to $\bP_1$).}
If $P_0\in\{0,1\}$, then $P_t\equiv P_0$ is constant and equal to $X$. In this case it is trivial to check that $(U,P,Y)$ is a separated control of the same value as $(U,X,Y)$. In the sequel, we assume that $0<P_0<1$. Then the measure $\bP$ can be conditioned on the event $X=x$, for all $x \in \bX$. This yields measures $\bP_x$ such that
\begin{align}
  	\bP_1(X=1)&=1, & \bP_0(X=0)&=1, & \bP &= P_0 \bP_1+(1-P_0)\bP_0.
\end{align}
The process $P/P_0$ is the $\{\cF^Y_t\}$-density process of $\bP_1$ relative to $\bP$ because for all $A\in \cF^Y_t$, 
\begin{equation}\label{equ:P_via_D}
    \int_A P_t \rd\bP
    =\int_A\bE[X|\cF^Y_t]\rd\bP
    =\int_AX \rd\bP=P_0 \bP_1(A) .
\end{equation}

\emph{Step 2 (Stochastic exponential relating the martingale problems under $\bP$ and $\bP_1$).} For each $f \in \cD(\cA)$, let $\ol{\cA}f$ be the average of $\cA f$ over $x\in\bX$ with weights $p$ and $(1-p)$,
\begin{equation}
    \ol{\cA}f(u,p,y) = p \cA f(u,1,y)+(1-p)\cA f(u,0,y).
\end{equation}
Let $f\in\cD(\cA)$ and set $g(x,y)=f(1,y)$. Then $g \in \cD(A)$ and $g$ is constant in $x \in \bX$. By \autoref{def:po:co}, the process
\begin{equation}
    g(1,Y)-g(1,Y_0)-\cA g(U,X,Y) \bullet I
\end{equation}
is a martingale under $\bP$. Taking $\{\cF^Y_t\}$-optional projections, one obtains that the process 
\begin{equation}\label{equ:Mf}
    M=g(1,Y)-g(1,Y_0)- \ol{\cA}g(U,P,Y) \bullet I
\end{equation}
is an $\{\cF^Y_t\}$-martingale under $\bP$. Moreover, as $X=1$ holds $\bP_1$-a.s., the process 
\begin{equation}\label{equ:Mtildef}
    \wt M = g(1,Y)-g(1,Y_0)-\cA g(U,1,Y) \bullet I
\end{equation}
is an $\{\cF^Y_t\}$-martingale under $\bP_1$. The difference between these two processes is given by
\begin{multline}\label{equ:m_minus_tilde_m}
M-\wt M
=\partial_y g(1,Y)\big(\beta(U,1,Y)-\ol\beta(U,P,Y)\big)
\bullet I
\\
+\int_\bY\big(g(1,Y+z)-g(1,Y)-\partial_y g(1,Y)\chi(z)\big)
\big(K(U,1,Y,\rd z)-\ol K(U,P,Y,\rd z)\big) 
\bullet I.
\end{multline}
For any $p>0$, let $\psi_1$ and $\psi_2$  be defined by
\begin{equation}
  	\psi_1(u,p,y)=(1-p)\phi_1(u,y), 
  	\quad 
  	\psi_2(u,p,y,z)=\frac{\phi_2(u,y,z)}{p \phi_2(u,y,z)+(1-p)\big(2-\phi_2(u,y,z)\big)},
\end{equation} 
where $\phi_1,\phi_2$ stem from \autoref{ass:gir}. Then the following relations hold for any $p>0$:
\begin{equation}\label{equ:psi}\begin{aligned}
    \be(u,1,y)-\ol\be(u,p,y)&=\si^2(u,y) \psi_1(u,p,y) 
    + \int_{\R^n} \big(\psi_2(u,p,y,z)-1 \big) \chi(z) \ol K(u,p,y,\rd z) ,
    \\
    K(u,1,y,\rd z) &= \psi_2(u,p,y,z) \ol K(u,p,y,\rd z)
\end{aligned}\end{equation}
For any $n\in\mathbb N$, let $T_n$ be the stopping time
\begin{equation}\label{equ:T_n}
  	T_n = \inf \{ t \geq 0 \colon P_t < 1/n \text{ or } P_{t-}<1/n \text{ or } |Y_t| > n \} 
  	\wedge n.
\end{equation}
Since $P>0$ holds on any interval $\llbracket 0,T_n\llbracket$, \autoref{equ:psi} can be used to rewrite \autoref{equ:m_minus_tilde_m} as
\begin{multline}\label{equ:m_minus_tilde_m2}
M^{T_n}-\wt M^{T_n}
=
\partial_y g(1,Y) \big(\si^2(U,Y) \psi_1(U,P,Y) \big)
\one_{\llbracket 0, T_n \rrbracket} \bullet I
\\
+\int_\bY\big(g(1,Y+z)-g(1,Y)\big)
\big(\psi_2(U,P,Y,z)-1\big)\ol K(U,P,Y,\rd z) 
\one_{\llbracket 0, T_n \rrbracket}\bullet I.
\end{multline}
Let $\mu$ be the jump measure of $Y$, $\nu$ its $\{\cF^Y_t\}$-compensator under $\bP$, and
\begin{equation}\label{equ:Ln}
  	L^n = \psi_1(U,P_-,Y_-)\one_{\llbracket 0, T_n \rrbracket} \bullet Y^c 
    + \big(\psi_2(U,P_-,Y_-,z)-1\big) \one_{\llbracket 0, T_n \rrbracket} 
    * \big(\mu-\nu\big)(\rd z,\rd t) .
\end{equation}
Then $L^n$ is a local $\{\cF^Y_t\}$-martingale under $\bP$. Keeping track of the terms in It\=o's formula the same way as in the proof of \citet[Theorem II.2.42]{Jacod} shows that
\begin{equation}
  	M = \partial_y g(1,Y_-) \bullet Y^c 
  	+ \big(g(1,Y_-+z)-g(1,Y_-)\big) * \big(\mu-\nu\big)(\rd z,\rd t)
\end{equation}
is the decomposition of $M$ into its continuous and purely discontinuous local martingale parts. It is now easy to calculate the predictable quadratic covariation of $M^{T_n}$ and $L^n$. Indeed, a comparison with \autoref{equ:m_minus_tilde_m2} shows that 
\begin{equation}
  	M^{T_n}-\wt M^{T_n} = \langle M^{T_n}, L^n \rangle.
\end{equation}
Equivalently, letting $D^n=\mathcal E(L^n)$ denote the stochastic exponential of $L^n$, 
\begin{equation}\label{equ:tilde_m_angle_bracket}
  	\wt M^{T_n} = M^{T_n} - \langle M^{T_n}, L^n \rangle 
  	= M^{T_n} - \frac{1}{D^n_-} \bullet \langle M^{T_n}, D^n \rangle.
\end{equation}

\emph{Step 3 (Martingale property of stochastic exponential).} We will show that the local martingale $D^n$ is a martingale by verifying the conditions of \citet[Th\'eor\`eme IV.3]{Lepingle}. For any $w \in [0,1]$, 
\begin{equation}\label{equ:p1p}
    \frac{p(1-p)}{p w+(1-p)(1-w)} 
    \leq \frac{p(1-p)}{p \wedge (1-p)} \leq 1
\end{equation}
holds because the nominator on the left-hand side is a convex combination of $p$ and $(1-p)$. Replacing $w$ by $\phi_2(u,y,z)/2$ in \autoref{equ:p1p} one obtains
\begin{equation}\label{equ:psi_and_phi}
  	\big(\psi_2(u,p,y,z)-1\big)^2 = 
    \left(\frac{2(1-p)(\phi_2(u,y,z)-1)}
    {p \phi_2(u,y,z)+(1-p)\big(2-\phi_2(u,y,z)\big)}\right)^2
    \leq \frac{1}{p^2}\big(\phi_2(u,y,z)-1\big)^2.
\end{equation}
This inequality relates the values of $\phi_2,\psi_2$ under the transformation $w\mapsto (w-1)^2$. It can equivalently be expressed in terms of the functions $w\mapsto w \log(w)-w+1$ or $w \mapsto 1-\sqrt{w(2-w)}$ because for all $w \in [0,2]$,
\begin{alignat}{2}
	\label{equ:y1}
  	w \log(w)-w+1 & \leq (w-1)^2 &&\leq 4 \big(w \log(w)-w+1\big), \\
    \label{equ:y2}
    1-\sqrt{w(2-w)} &\leq (w-1)^2 &&\leq 2 \left(1-\sqrt{w(2-w)}\right).
\end{alignat}
Actually, the first inequality in \autoref{equ:y1} holds for all $w\geq 0$,
which implies that
\begin{multline}\label{equ:bound1}
  	\int \Big(\psi_2(u,p,y,z) \log\big(\psi_2(u,p,y,z)\big)
  	-\psi_2(u,p,y,z)+1\Big) \ol K(u,p,y,\rd z) 
  	\\
    \leq \int \big(\psi_2(u,p,y,z)-1\big)^2 \ol K(u,p,y,\rd z) 
    \leq \frac{1}{p^2} \int \big(\phi_2(u,y,z)-1\big)^2 \ol K(u,p,y,\rd z)
    \\
    \leq \frac{2}{p^2} \int \left(1-\sqrt{\phi_2(u,y,z)\big(2-\phi_2(u,y,z)\big)}\right) 
    \ol K(u,p,y,\rd z) 
    \\
    \leq \frac{4}{p^2} \int \left(1-\sqrt{\phi_2(u,y,z)\big(2-\phi_2(u,y,z)\big)}\right) 
    \ol K(u,\tfrac12,y,\rd z) .
\end{multline}
By \autoref{ass:novikov}, this expression is bounded as long as $p$ stays away from zero. Moreover, by the same assumption, the following expression is bounded:
\begin{equation}\label{equ:bound2}
  	\psi_1(u,p,y)^\top \sigma^2(u,y)\psi_1(u,p,y) 
  	=(1-p)^2 \phi_1(u,y)^\top \sigma^2(u,y)\phi_1(u,y).
\end{equation}
Therefore, 
\begin{equation}
\bE\left[\exp\left(\frac12 \langle L^{n,c},L^{n,c}\rangle_\infty
+\big((1+z)\log(1+z)-z\big) * \nu^{T_n}_\infty \right)\right] <\infty,
\end{equation}
which is the condition of \citet[Th\'eor\`eme IV.3]{Lepingle} implying that $D^n=\mathcal E(L^n)$ is a uniformly integrable martingale. Therefore, $D^n_{T_n} \bP$ is a probability measure. 

\emph{Step 4 (Identification of stochastic exponential and filter).} By \autoref{def:po:co}, $U=F(Y)$ for a process $F$ on $D_\bY[0,\infty)$. We will use the well-posedness of the martingale problem for $(\cA,F)$ to show that $D^n_{T_n}\bP$ agrees with $\bP_1$ on $\cF^Y_{T_n}$. By Girsanovs' theorem 
and \autoref{equ:tilde_m_angle_bracket}, $\wt M^{T_n}$ is an $\{\cF^Y_t\}$-martingale under $D^n_{T_n} \bP$. The process $\wt M$ can be written as
\begin{equation}
\wt M=f(1,Y)-f(1,Y_0)-\cA f(U,1,Y)\bullet I
\end{equation}
because $\cA$ has no derivatives or non-local terms in the $x$-direction. As $f \in \cD(\cA)$ was chosen arbitrarily, the tuple $(1,Y)$ under the measure $D^n_{T_n}$ solves the martingale problem for $(\cA,F)$ stopped at $T_n$. The same can be said about the tuple $(X,Y)$ under the measure $\bP_1$. Moreover, the distribution of $(1,Y_0)$ under $D^n_{T_n}\bP$ coincides with the distribution of $(X,Y_0)$ under $\bP_1$. According to \autoref{def:po:co}, local uniqueness holds for the martingale problem. It follows that $D^n_{T_n} \bP$ coincides with $\bP_1$ on $\cF^Y_{T_n}$. The characterization of $P/P_0$ as the density process of the measure $\bP_1$ relative to $\bP$ obtained in Step~1 implies that $P=P_0 D^n_{T_n}$ holds on $\llbracket 0,T_n\rrbracket$.

\emph{Step 5 (Filter solves the martingale problem with generator $\cG$).}
To show that $(U,P,Y)$ is a separated control, one has to prove that for any $f \in \cD(\cG)$, 
\begin{equation*}
	N = f(P,Y)-f(P_0,Y_0)-\cG f(U,P,Y)\bullet I
\end{equation*}
is an $\{\cF^Y_t\}$-martingale under $\bP$. On the interval $\llbracket 0, T_n \rrbracket$, $P$ agrees with $P_0 D^n$ and consequently satisfies $P = P_- \bullet L^n$. Therefore, the jumps of $P$ on this interval are
\begin{equation}
  	\De P = P_- \; \De L^n 
  	= P_- \big(\psi_2(U,P_-,Y_-,\De Y)-1 \big) \one_{\De Y \neq 0} 
  	= j(U,P_-,Y_-,\De Y) \one_{\De Y \neq 0} ,
\end{equation}
where the function $j$ is defined in \autoref{ass:G}. Moreover, on the same interval $\llbracket 0, T_n \rrbracket$,
\begin{equation}\begin{aligned}
\langle P^c,P^c\rangle
&= P_-^2 \bullet \langle L^{n,c},L^{n,c} \rangle 
= \sum_{i,j} P_-^2 \psi_{1,i}(U,P_-,Y_-) \psi_{1,j}(U,P_-,Y_-)\bullet 
\langle Y^{i,c},Y^{j,c} \rangle 
\\
&= P^2(1-P)^2 \phi_1(U,Y)^\top \si(U,Y)^2 
\phi_1(U,Y) \bullet I
\\
\langle P^c,Y^c\rangle
&= P_-^2 \bullet \langle L^{n,c},Y^c \rangle 
= P_-^2 \psi_1(U,P_-,Y_-)^\top \bullet \langle Y^c,Y^c \rangle 
= P^2 \psi_1(U,P,Y)^\top \sigma^2(U,Y) \bullet I 
\\
\langle Y^c,Y^c \rangle&=\sigma^2(U,Y)\bullet I.
\end{aligned}\end{equation}
It follows from It\=o's formula and the definition of $\cG$ in \autoref{ass:G} that the stopped process $N^{T_n}$ is an $\{\cF^Y_t\}$-local martingale under $\bP$. It is also bounded by \autoref{ass:bou}, so it is a martingale. Setting $g(x,y)=f(0,y)$, one has $g\in\cD(\cA)$, and the process
\begin{equation}
M=g(0,Y)-g(0,Y_0)-\ol\cA g(U,P,Y) \bullet I
\end{equation}
is a martingale. Then it holds for any bounded stopping time $S$ and each $n \in \bN$ that
\begin{equation}
    \bE\big[N_S\big] 
    = \bE \big[N_{S \wedge T_n} + N_{S \vee T_n} - N_{T_n}\big]
    = \bE\big[N_{S \wedge T_n} 
    + M_{S \vee T_n} - M_{T_n}+ R_n\big]
    = \bE\big[R_n\big],
\end{equation}
with a remainder $R_n$ given by
\begin{equation}
R_n=\big(N_{S \vee T_n} - N_{T_n}\big)
-\big( M_{S \vee T_n} - M_{T_n}\big).
\end{equation}
Let $\omega \in \Omega$ and
\begin{equation}\label{equ:T}
	T=\lim_{n\to\infty}	T_n 
	=\inf \{t\geq 0\colon P_{t}=0 \text{ or } P_{t-}=0\}.
\end{equation}
If $P_{T-}(\omega)=0$, then $T_n(\omega)<T(\omega)$ holds for all $n \in \bN$. Otherwise, there is $k\in\bN$ such that $T_n(\omega)=T(\omega)$ holds for all sufficiently large $n$. Therefore, \begin{equation}
	\lim_{n\to\infty} R_n = \left\{
	\begin{aligned}
		&(N_{T-}-N_{T-})-(M_{T-}-M_{T-}),
		&&\text{if } P_{T-}=0 \text{ and } t<T,
		\\
		&(N_S-N_{T-})-(M_S-M_{T-}),
		&&\text{if } P_{T-}=0 \text{ and } t\geq T,
		\\
		&(N_{S\vee T}-N_{T})-(M_{S\vee T}-M_{T}),
		&&\text{if } P_{T-}\neq 0.
	\end{aligned}\right.
\end{equation}
It can be seen from the definitions of $\cA$ and $\cG$ in \autoref{ass:A,ass:G} that $\cG f(u,0,y)=\ol\cA g(u,0,y)$. Therefore, $N=M$ holds on the interval $\llbracket T,\infty\llbracket$, where $P=0$. Moreover, $N_{T-}=M_{T-}$ holds if $P_{T-}=0$. This implies that $\lim_{n\to\infty}R_n=0$. The processes $M^S$ and $N^S$ are bounded, which follows from \autoref{ass:bou} and the boundedness of $S$. Therefore,
\begin{equation}
	R_n=\one_{T_n<S}\big((N_S-N_{S\wedge T_n})
	-(M_S-M_{S\wedge T_n})\big)
\end{equation}
is bounded by a constant not depending on $n$. By the dominated convergence theorem, $\bE[N_S]=\lim_{n\to\infty}\bE[R_n]=0$. As this holds for all bounded stopping times $S$, we conclude that $N$ is a martingale. As $f \in \cD(\cG)$ was chosen freely, $(U,P,Y)$ is a separated control.

\emph{Step 6 (Value of separated control).} $(U,P,Y)$ has the same value as $(U,X,Y)$ because $\ol b(U,P,Y)$ is the $\{\cF^Y_t\}$-optional projection of $b(U,X,Y)$. 
\end{proof}

\begin{lemma}[Approximation]\label{lem:se:app}
Separated controls can be approximated arbitrarily well in value by separated step controls:
\begin{equation}\label{equ:lem:se:app}
  	V^\se(p,y)=\sup_\de V^{\se,\de}(p,y).
\end{equation}
Here $V^{\se,\de}$ denotes the value function obtained by admitting only processes $U$ which are piecewise constant on an equidistant time grid of step size $\delta>0$ in the separated control problem. 
\end{lemma}

\begin{proof}
\emph{Step 1 (Filtered martingale problem).} 
Let $U$ be deterministic and let $(P,Y)$ be a c\`adl\`ag process with values in $[0,1]\x\bY$. We identify $P$ with the $\cP(\bX)$-valued process $\Pi$ given by
\begin{equation}\label{equ:p_abuse_of_notation}
\Pi_t(\rd x) = P_t \delta_1(\rd x) + (1-P_t)\delta_0(\rd x),
\end{equation}
where $\delta_x$ denotes the Dirac measure at $x\in\bX$. In line with \citet{KurtzOcone}, we say that $(\Pi,Y)$ is a solution of the filtered martingale problem for $(\cA,U)$ if
\begin{equation}\label{equ:filtered_mgp}
\int_\bX f(x,Y) \Pi(\rd x)
-\int_\bX f(x,Y_0) \Pi_0(\rd x)
-\int_\bX \cA f(U,x,Y) \Pi(\rd x) \bullet I
\end{equation}
is a martingale, for each $f \in \cD(\cA)$, and $\Pi$ is $\{\cF^Y_t\}$-adapted.

We will use \citet[Theorem 3.6]{KurtzNappo} to show that uniqueness holds for the filtered martingale problem. Thus, we have to verify points (i)-(vi) of Condition 2.1 in this paper. These are conditions on the operator $\cA f(U,x,y)$ in \eqref{equ:filtered_mgp}, interpreted as a time-dependent generator of $(X,Y)$. To put everything into a time-homogeneous framework, we work with the time-augmented process $(I,X,Y)$. Its generator $\cA^U$ is given by
\begin{align}
\cD(\cA^U)=C^2_b(\R\x\bX\x\bY),
&&
\cA^U g(t,x,y)
=\partial_t g(t,x,y)+\cA g_t(U_t,x,y),
\end{align}
where $g_t(x,y)=g(t,x,y)$.
For point (i), there is nothing to prove. For point (ii), one has to show that $\cA f(u,x,y)$ is continuous in $(u,x,y)$, for each $f\in \cD(\cA)$. To see this, let 
\begin{align}
m(u,x,y)&=1+\int_{\R^d}\big(\lvert z\rvert^2\wedge 1\big) K(u,x,y,\rd z), 
\\
\hat f(u,x,y,z)&=m(u,x,y)\frac{f(x,y+z)-f(x,y)-\partial_y f(x,y)\chi(z)}{\lvert z\rvert^2\wedge 1}, 
\\
\hat K(u,x,y,\rd z)&=\frac{\big(\lvert z\rvert^2\wedge 1\big) K(u,x,y,\rd z)}{m(u,x,y)}.
\end{align}
Then everything is set up such that
\begin{equation}\label{equ:jump_term}
\int_\bY \big(f(x,y+z)-f(x,y)-\partial_y f(x,y)\chi(z)\big) K(u,x,y,\rd z)
= \int_\bY \hat f(u,x,y,z) \hat K(u,x,y,\rd z).
\end{equation}
Now let $(u_n,x_n,y_n)_{n\in\mathbb N}$ be a sequence in $\bU\x\bX\x\bY$ converging to $(u,x,y)$. By \autoref{ass:con}, $m(u,x,y)$ is continuous and the measures $\hat K(u_n,x_n,y_n,\rd z)$ are weakly convergent. A version of Skorokhod's representation theorem for measures instead of probability measures (for example \citet{Startek2012}) implies that there are mappings $(Z_n)_{n\in\mathbb N}$ and $Z$ with values in $\bY$, all defined on the same measure space with finite measure, such that for each $n \in \mathbb N$, $Z_n$ has distribution $\hat K(u_n,x_n,y_n,\rd z)$, $Z$ has distribution $\hat K(u,x,y,\rd z)$, and $Z_n\to Z$ almost surely. By the dominated convergence theorem, 
\begin{equation}
\bE\big[\hat f(u_n,x_n,y_n,Z_n)\big] \to \bE\big[\hat f(u,x,y,Z)\big],
\end{equation}
which shows that the expression in \eqref{equ:jump_term} is continuous in $(u,x,y)$. This settles point (ii). Point (iii) is satisfied with $\psi=1$ by \autoref{ass:bou}. Points (iv) and (vi) are satisfied for $\cD(\cA^U)=C^2_b(\R\x\bX\x\bY)$. Finally, point (v) is satisfied because of \autoref{ass:wel}, which guarantees that for each constant, deterministic control $U$ and all initial conditions, there exists a c\`adl\`ag solution of the martingale problem for $\cA^U$ (cf. the discussion before Theorem 2.1 in \citet{Kurtz1998}). Moreover, by \autoref{ass:wel}, uniqueness holds for the martingale problem for $\cA^U$, which coincides with the martingale problem for $(\cA,U)$ from \autoref{def:mgp_A}. Thus, all conditions of \citet[Theorem 3.6]{KurtzNappo} are fulfilled and uniqueness holds for the filtered martingale problem.

\emph{Step 2 (Projecting separated controls to solutions of the filtered martingale problem).} 
Let $(U,P,X)$ be a separated control with deterministic control process $U$. Let $f\in\cD(\cG)$ be affine in the first variable $p$, i.e., 
\begin{equation}
f(p,x)=p f(1,x)+(1-p)f(0,x).
\end{equation}
Then $\cG f(u,p,x)$ is also affine in $p$, i.e., 
\begin{equation}
\cG f(u,p,x)
=p \cG f(u,0,x)+(1-p)\cG f(u,0,x)
=p \cA f(u,0,x)+(1-p)\cA f(u,0,x).
\end{equation}
This can be verified using the definition of $\phi_1$ and $\phi_2$, noting that all quadratic terms in $p$ cancel out in the expression of $\cG f(u,p,x)$. Identifying $P$ with $\Pi$ as in Step 1, one obtains that the process 
\begin{multline}
f(P,Y)-f(P_0,Y_0)-\cG f(U,P,Y) \bullet I
\\=
\int_\bX f(x,Y) \Pi(\rd x)
-\int_\bX f(x,Y_0) \Pi_0(\rd x)
-\int_\bX \cA f(U,x,Y) \Pi(\rd x) \bullet I
\end{multline}
is a martingale. If $\wt \Pi$ denotes the $\{\cF^Y_t\}$-optional projection of $\Pi$, then 
\begin{equation}
\int_\bX f(x,Y) \wt\Pi(\rd x)
-\int_\bX f(x,Y_0) \wt\Pi_0(\rd x)
-\int_\bX \cA f(U,x,Y) \wt\Pi(\rd x) \bullet I
\end{equation}
is also a martingale. Thus, $(\tilde \Pi,Y)$ is a solution of the filtered martingale problem. By the previous step, the law of $(\tilde \Pi,Y)$ is uniquely determined. An important consequence is that all separated controls sharing the same deterministic control process $U$ have the same value:
\begin{equation}\label{equ:j_of_u}
\bE\left[\int_0^\infty \rho e^{-\rho t}
\ol b(U_t,P_t,Y_t)\rd t\right]	
=
\bE\left[\int_0^\infty \rho e^{-\rho t}
\ol b(U_t,\wt P_t,Y_t)\rd t\right]
=:J(U),
\end{equation}
where $\wt P$ is the $\{\cF^Y_t\}$-optional projection of $P$.

\emph{Step 3 (Tightness of separated controls).} Let $(U^n,P^n,Y^n)$ be separated controls with deterministic control processes $U^n$ and let $U^n \to U$ in the stable topology, i.e., 
\begin{equation}
	\int_0^\infty g(U^n_t,t)\rd t \to 
	\int_0^\infty g(U_t,t)\rd t
\end{equation}
for all bounded measurable functions $g\colon \bU\x[0,\infty)\to\R$ with compact support which are continuous in $u$. The stable topology coincides with the vague topology, checked on continuous functions with compact support. For more details on the vague and stable topology we refer to \citet{ElKaroui1988} and \citet{jacod1981type}. We will use \citet[Theorem IX.3.9]{Jacod} to show that the laws of $(P^n,Y^n)$ are tight. Thus, we have to verify the conditions of this theorem. By the same estimates as in Step 3 of the proof of \autoref{lem:fi}, one obtains that
\begin{equation}\begin{aligned}&
\int_\bY \left(\lvert j(u,p,y,z)\rvert^2+\vert z\rvert^2\right) 
\wedge 1\ \ol K(u,p,y,\rd z)
\\&\qquad
\leq 2\int_\bY j(u,p,y,z)^2\ \ol K(u,p,y,\rd z)
+2\int_\bY \lvert z\rvert^2\wedge 1\ \ol K(u,p,y,\rd z)
\\&\qquad
\leq 2\int_\bY \big(\phi_2(u,y,z)-1\big)^2\ \ol K(u,p,y,\rd z)
+2\int_\bY \lvert z\rvert^2\wedge 1\ \ol K(u,p,y,\rd z)
\\&\qquad
\leq 4\int_\bY \left(1-\sqrt{\phi_2(u,y,z)\big(2-\phi_2(u,y,z)\big)}\right)  
\ol K(u,p,y,\rd z)
+2\int_\bY \lvert z\rvert^2\wedge 1\ \ol K(u,p,y,\rd z).
\end{aligned}\end{equation}
By \autoref{ass:bou,ass:novikov}, the integrals on the last line above are bounded by a constant which does not depend on $(u,p,y)$. By \autoref{ass:bou,ass:novikov}, also the drift and the diagonal entries of the volatility matrix
\begin{align}
\ol\beta(u,p,y),
&&
p^2(1-p)^2 \phi_1(u,y)^\top\sigma^2(u,y)\phi_1(u,y), 
&&
\sigma^2(u,y)
\end{align}
are bounded by a constant not depending on $(u,p,y)$. It follows that Condition IX.3.6 (the strong majoration hypothesis) is satisfied.  Condition IX.3.7 (the condition on the big jumps) follows from \autoref{ass:bigjumps}. By the stable convergence of $U^n$ to $U$, using \autoref{ass:con} and the bounds which were just shown, the following convergence holds for all $t\geq 0, (P,Y) \in D_{[0,1]\x\bY}[0,\infty)$, and functions $g \in C_b([0,1]\x\bY)$ vanishing near the origin:
\begin{equation}\begin{aligned}
\ol\beta(U^n,P,Y)\bullet I_t 
&\to \ol\beta(U,P,Y) \bullet I_t,
\\
P^2(1-P)^2 \phi_1(U^n,Y)^\top\sigma^2(U^n,Y)\phi_1(U^n,Y)\bullet I_t
&\to P^2(1-P)^2 \phi_1(U,Y)^\top\sigma^2(U,Y)\phi_1(U,Y)\bullet I_t
\\
P(1-P)\sigma^2(U^n,Y)\phi_1(U^n,Y)\bullet I_t
&\to P(1-P)\sigma^2(U,Y)\phi_1(U,Y) \bullet I_t
\\
\sigma^2(U^n,Y)\bullet I_t
&\to \sigma^2(U,Y) \bullet I_t,
\\
\int_\bY g\big(j(U^n,P,Y,z),z\big)\,\ol K(U^n,P,Y,\rd z) \bullet I_t	
&\to \int_\bY g\big(j(U,P,Y,z),z\big)\,\ol K(U,P,Y,\rd z) \bullet I_t.
\end{aligned}
\end{equation}
It follows from Lemma IX.3.4 that the conditions of Theorem IX.3.9 are satisfied. Thus, the laws of $(P^n,Y^n)$ are tight. Moreover, any limit $(P,Y)$ of a weakly converging subsequence of $(P^n,Y^n)$ solves the martingale problem for $(\cG,U)$ and defines a separated control $(U,P,Y)$. This follows from \citet[Theorem IX.2.11]{Jacod} by the same assumptions. 

\emph{Step 4 (Step controls).}
For any $\delta>0$, the mapping
\begin{align}
\Psi^\delta\colon L_\bU[0,\infty) \to L_\bU[0,\infty),
&&
(\Psi^\delta U)_t=\sum_{i=0}^\infty U_{i\delta} \one_{(i\delta,(i+1)\delta]}(t)
\end{align}
approximates deterministic control processes by step control processes of step size $\delta$. Indeed, $\lim_{\delta\to 0}(\Psi^\delta U)_t = U_t$ holds for each $t\geq 0$. Moreover, by dominated convergence, $\Psi^\delta U$ converges stably to $U$. Let
\begin{equation}
L^0_{\bU}[0,\infty)
=\bigcup_{\delta>0}\left\{\Psi^{\delta}U\colon U \in L_{\bU}[0,\infty)\right\}
\subset L_{\bU}[0,\infty)
\end{equation}
denote the set of all step control processes. For any step control process $U \in L^0_{\bU}[0,\infty)$, there is a control with partial observations $(U,X,Y)$ by \autoref{ass:wel} and a corresponding separated control $(U,P,Y)$ by \autoref{lem:fi}. Let $\bQ_U$ denote the law of $(P,Y)$ under $U$.  If $U^n \in L^0_{\bU}[0,\infty)$ converges stably to a step control $U \in L^0_{\bU}[0,\infty)$, then $\bQ_{U^n}$ converges weakly to $\bQ_U$ by the arguments in Step 2 and by \autoref{ass:wel2} ensuring uniqueness of the martingale problem for $(\cG,U)$. As continuity implies measurability, $\bQ$ is a transition kernel from $L^0_{\bU}[0,\infty)$ with the Borel sigma algebra of stable convergence to Skorokhod space $D_{[0,1]\x\bY}[0,\infty)$. 

\emph{Step 5 (Approximation of deterministic controls).} 
Let $U \in L_\bU[0,\infty)$ and define $U^n=\Psi^{1/n}U$, for each $n\in\bN$. By \autoref{ass:wel} there are controls with partial observations $(U^n,X^n,Y^n)$ and by \autoref{lem:fi} corresponding separated controls $(U^n,P^n,Y^n)$. By the tightness result of Step 3, any subsequence along which $J(U^n)$ converges contains another subsequence, still denoted by $n$, such that $(P^n,Y^n)$ converge weakly to some solution $(P,Y)$ of the martingale problem for $(\cG,U)$. By Skorokhod's representation theorem we may assume after passing to yet another subsequence that $(P^n,Y^n)$ and $(P,Y)$ are defined on the same probability space and that $(P^n,Y^n)$ converge to $(P,Y)$ almost surely. As $\Delta P_t = 0$ holds almost surely for each fixed $t\geq 0$, it follows from the dominated convergence theorem and the pointwise convergence of $U^n_t$ to $U_t$ that
\begin{equation}\begin{aligned}
\lim_{n\to\infty}J(U^n)
=
\int_0^\infty \rho e^{-\rho t} \lim_{n\to\infty}
\bE\left[\ol b(U^n_t,P^n_t,Y^n_t)\right]\rd t
=
\int_0^\infty \rho e^{-\rho t} 
\bE\left[\ol b(U,P_t,Y_t)\right]\rd t
=J(U).
\end{aligned}\end{equation}

\emph{Step 5 (Approximation of arbitrary controls).} 
The law of any separated control $(U,P,Y)$ is a probability measure $\bP$ on the space $L_{\bU}[0,\infty) \x D_{[0,1]\x\bY}[0,\infty)$. We will work on this canonical probability space in the sequel. Using disintegration, $\bP$ can be written in the form 
\begin{equation}\label{equ:se:disintegration}
    \bP(\rd U,\rd P,\rd Y)=\bP(\rd U)\bP_U(\rd P,\rd Y).
\end{equation}
Accordingly, the value of the control can be expressed as 
\begin{equation}
J^\se(U,P,Y) 
=
\int_{L_{\bU}[0,\infty)}
\bE_{\bP_U}\left[\int_0^\infty \rho e^{-\rho t} 
\ol b(U_t,P_t,Y_t)\rd t\right]
\bP(\rd U)
\end{equation}
For $\bP$-a.e. $U$, the process $(P,Y)$ under the measure $\bP_U$ solves the martingale problem for $(\cG,U)$. Moreover, the process $U$ is deterministic under the measure $\bP_U$. By Step 2, all solutions of the martingale problem $(\cG,U)$ with deterministic control process $U$ have the same value $J(U)$. This allows one to express the value of the control as 
\begin{equation}
J^\se(U,P,Y) = \int_{L_{\bU}[0,\infty)} J(U) \bP(\rd U).
\end{equation}
By Step 4 and dominated convergence, 
\begin{equation}
J^\se(U,P,Y) = 
\lim_{n\to\infty}
\int_{L_{\bU}[0,\infty)} J(\Psi^{1/n}U) \bP(\rd U)
=
\lim_{n\to\infty} J^\se(U^n,P^n,Y^n),
\end{equation}
where $(U^n,P^n,Y^n)$ is the coordinate process on $L_{\bU}[0,\infty)\x D_{[0,1]\x\bY}[0,\infty)$ under the measure $\bQ_{\Psi^{1/n}U}(\rd P,\rd Y)\bP(\rd U)$. Thus, $(U^n,P^n,Y^n)$ is a sequence of separated step controls approximating $(U,P,Y)$ in value. 
\end{proof}

\begin{lemma}[From separated to partially observed controls]\label{lem:se_to_po}
For every separated step control, there exists a step control with partial observations of at least the same value, implying $V^{\po,\de}(p,y) \geq V^{\se,\de}(p,y)$.
\end{lemma}
  
\begin{proof}
\emph{Step 1 (Reduction to Markovian step controls).}
To distinguish the separated and the partially observed versions of the problem, we will mark objects of the separated problem with a tilde. By \autoref{ass:wel2}, the discretized separated problem is that of controlling the Markov chain $(\wt P_{t_i},\wt Y_{t_i})$, where $(t_i)_{i \in\bN}$ is a uniform time grid of step size $\de>0$. It is well-known that optimal Markov controls exist for such problems (see e.g. \citet{Berry,Seierstad}). We will prove the lemma by showing that every Markov control for the discretized, separated problem corresponds to a step control for the problem with partial observations which has the same value. So we start with a Markovian step control $(\wt U,\wt P,\wt Y)$ with control process $\wt U$ given by
\begin{align}
    \wt U_t = F_i(\wt P_{t_i},\wt Y_{t_i}), 
    &&
    \text{if $t\in (t_i,t_{i+1}]$,}
\end{align} 
for some functions $F_i\colon [0,1]\x \bY \to \bU$, $i\in\bN$.

\emph{Step 2 (Construction of a candidate control with partial observations).}
To construct the control for the problem with partial observations, we work on the canonical space $\Om=\bX\x D_\bY[0,\infty)$ with its natural sigma algebra and filtration. The coordinates on this space are denoted by $(X,Y)$. When $T$ is a (strict) stopping time, $\bP$ is a probability measure on $\Om$, and $\bQ$ is an $\cF_T$-measurable random variable with values in the space of probability measures on $\Om$, then we let $\bP\otimes_T \bQ$ denote the unique probability measure on $\Om$ such that (i) the law of the stopped process $(X,Y^T)$ is equal to $\bP$ on the sigma algebra $\cF_T$ and (ii) the $\cF_T$-conditional law of the time-shifted process $(X,Y_{T+t})_{t\geq0}$ is $\bQ$. This notation is explained and relevant results are proven in \citet[6.1.2, 6.1.3 and 1.2.10]{StroockVaradhan} for continuous processes. For processes with jumps, the relevant results are \citet[Lemmas III.2.43-48]{Jacod}, but the notation $\bP\otimes_T \bQ$ is not used there.

By \autoref{ass:wel}, we get for each $(u,x,y)\in \bU\x\bX\x\bY$ a unique probability measure $\bQ^u(x,y)$ on $\Om$ such that $X=x$ and $Y_0=y$ holds almost surely and such that $(X,Y)$ solve the martingale problem for $(\cA,u)$ under $\bQ^u(x,y)$. By \citet[Theorem IX.3.39]{Jacod}, $\bQ^u(x,y)$ is weakly continuous, thus measurable, in $(u,x,y)$. Verifying the conditions of the theorem can be done as in the proof of \autoref{lem:se:app}, but it is easier in the present situation. We now define inductively for each $n\in\mathbb N$ a probability measure $\bP^n$ and a c\`adl\`ag process $P^n$ on $\Om$ as follows. 
\begin{equation}\begin{aligned}
    \bP^0&=P_0\bQ^{F_0(P_0,Y_0)}(1,Y_0)+(1-P_0)\bQ^{F_0(P_0,Y_0)}(0,Y_0), 
    & P^0_t&=\bE_{\bP^0}[X\mid\cF^Y_t], 
    \\
    \bP^n&=\bP^{n-1}\otimes_{t_n}\bQ^{F_n(P^{n-1}_{t_n},Y_{t_n})}(X,Y_{t_n}), 
    & P^n_t&=\bE_{\bP^n}[X\mid\cF^Y_t].
\end{aligned}\end{equation}
It follows that the measures $\bP^n$ and $\bP^m$ agree on $\cF_{t_n \wedge t_m}$ and that the processes $P^n$ and $P^m$ agree almost surely on $[0,t_n\wedge t_m]$. Therefore, there is a unique measure $\bP$ which coincides with $\bP^n$ on $\cF_{t_n}$, for all $n$. Furthermore, there is a unique c\`adl\`ag process $P$ that is almost surely equal to $P^n$ on $[0,t_n]$, for all $n$. If $U$ is defined as
\begin{equation}\label{equ:U_as_F}
    U_t = \sum_{i=0}^\infty F_i(P_{t_i},Y_{t_i})\one_{(t_i,t_{i+1}]}(t), 
\end{equation}
then by construction, the process
\begin{equation}
	f(X,Y)-f(X,Y_0)-\cA f(U,X,Y)\bullet I
\end{equation}
is a martingale, for each $f\in\cD(\cA)$ (see also \citet[Lemma III.2.48]{Jacod}).

\emph{Step 3 (Verification of the well-posedness condition).}
As $P$ is the $\{\cF^Y_t\}$-optional projection of $X$, it is indistinguishable from $G(Y)$ for some c\`adl\`ag process $G$ on $D_\bY[0,\infty)$ by \citet{delzeith2004skorohod}. It follows from \autoref{equ:U_as_F} that $U$ is indistinguishable from $F(Y)$ for some c\`agl\`ad process $F$ on $D_\bY[0,\infty)$. The martingale problem $(\cA,F)$ is well-posed by \autoref{ass:wel} because $F$ is a step process. Thus, the well-posedness condition of \autoref{def:po:co} is satisfied. 

\emph{Step 4 (Value of the control with partial observations).}
The process $P$ defined in Step 2 is the $\{\cF^Y_t\}$-optional projection of $X$. By \autoref{lem:fi}, $(U,P,Y)$ defines a separated control of the same value as $(U,X,Y)$. \autoref{ass:wel2} implies that $(U,P,Y)$ is equal in law to $(\wt U,\wt P,\wt Y)$. Therefore, $(U,X,Y)$ has the same value as $(\wt U,\wt P,\wt Y)$. 
\end{proof}

\section{Proofs of \autoref{sec:res}.}\label{sec:app:res}

The setup of \autoref{sec:res:set}, including Assumptions \ref{ass:A}--\ref{ass:mon}, holds.

\begin{lemma}[Payoff function]\label{lem:r}
For any control $(U,X,H,R)$ of the problem with partial observations,
\begin{equation}\label{equ:payoff}
\bE\left[\int_0^\infty \rho e^{-\rho t} \rd R_t\right] 
= 
\bE\left[\int_0^\infty \rho e^{-\rho t} b(U_t,X,H_t) \rd t\right],
\end{equation}
where $b$ is given by \autoref{ass:b}.
\end{lemma}

\begin{proof}
By the integrability condition on $K_R$ in \autoref{ass:str}, the process $R$ is a special semimartingale. Its canonical decomposition is 
\begin{equation}
	R=R_0+b(U,X,H) \bullet I + R^c + r*(\mu-\nu),
\end{equation}
where $\mu$ is the integer-valued random measure associated to the jumps of $R$ and $\nu=\one_{U=1} K_R(X,H,\cdot)$ is the compensator of $\mu$. For $\zeta_t=\rho e^{-\rho t}$ one obtains that 
\begin{equation}\label{equ:Rminusb}
	\zeta\bullet R-\zeta b(U,X,H) \bullet I 
	= \zeta\bullet R^c+\zeta r*(\mu-\nu)
\end{equation} 
is a local martingale. \autoref{equ:payoff} holds if it is a true martingale. 

Let $\chi_R(r)=\chi(0,r)$. The processes $\zeta\bullet R^c$ and $\zeta\chi_R(r)*(\mu-\nu)$ are square integrable martingales by the Burkholder-Davis-Gundy inequality because their quadratic variations are integrable:
\begin{equation}\begin{aligned}
	\bE\big[[\zeta\bullet R^c]_\infty\big]
	&=
	\bE\left[\zeta^2 \one_{U=1}\sigma_R(H)^2 \bullet I_\infty\right]<\infty,
	\\
	\bE\big[[\zeta\chi_R(r)*(\mu-\nu)]_\infty\big]&=
	\bE\big[\zeta^2\chi_R(r)^2*\mu_\infty\big]
	=\bE\big[\zeta^2\chi_R(r)^2*\nu_\infty\big]<\infty.
\end{aligned}\end{equation}
This follows from the bounds on $\sigma_R$ and $K_R$ in \autoref{ass:bou,ass:str}. Furthermore, the process $(r-\chi_R(r))*(\mu-\nu)$ is a uniformly integrable martingale on $[0,t]$ because it is of integrable variation:
\begin{equation}\begin{aligned}
	\bE\big[\on{Var}(\zeta(r-\chi_R(r))*(\mu-\nu))_\infty\big]
	&\leq
	\bE\left[\zeta|r-\chi_R(r)|*\mu_\infty\right]
	+\bE\left[\zeta|r-\chi_R(r)|*\nu_\infty\right]
	\\&=
	2*\bE\left[\zeta|r-\chi_R(r)|*\nu_\infty\right]<\infty.
\end{aligned}
\end{equation}
This follows from the bound on $K_R$ in \autoref{ass:str}. Therefore, the process in \autoref{equ:Rminusb} is a martingale, and \autoref{equ:payoff} holds. 
\end{proof}

\begin{lemma}[Elimination of the state variable $r$]\label{lem:no_r}
The value functions $V(p,h,r)$, $V^\delta(p,h,r)$ do not depend on $r$ and can be written as $V(p,h)$, $V^\delta(p,h)$.
\end{lemma}

\begin{proof}
For any $s \in \R$ and $f \in \cD(\cA)$, let $f_s(x,h,r)=f(x,h,r+s)$. Then $f_s \in \cD(\cA)$, and by \autoref{ass:str}, $\cA f(u,x,h,r+s) = \cA f_s(u,x,h,r)$. If $(U,X,H,R)$ is a control with partial observations, then the equation
\begin{multline}
f(X,H,R+s)-f(X,H_0,R_0+s)-\cA f(U,X,H,R+s) \bullet I
\\=f_s(X,H,R)-f_s(X,H_0,R_0)-\cA f_s(U,X,H,R) \bullet I
\end{multline}
shows that $(U,X,H,R+s)$ is also a control with partial observations. Moreover, the two controls have the same value. The same argumentation applies to separated controls. 
\end{proof}


\begin{lemma}\label{lem:mon}
The value functions $V(p,h)$ and $V^\delta(p,h)$ are convex, non-decreasing in $p$, and non-decreasing in $h$. 
\end{lemma}

\begin{proof}
Recall from Step 5 in the proof of \autoref{lem:se:app} that the value of any separated control can be written as
\begin{equation}
J^\se(U,P,H,R) = \int_{L_{\bU}[0,\infty)} J(U) \bP(\rd U),
\end{equation}
where $\bP(\rd U)$ is the marginal distribution of $U \in L_\bU[0,\infty)$ and $J(U)$ is the value of a deterministic control process $U$. In the definition of $J(U)$ in \autoref{equ:j_of_u}, $P$ is a martingale and $(U,H)$ are deterministic. Therefore, 
\begin{multline}
J(U)=\bE\left[\int_0^\infty \rho e^{-\rho t} \ol b(U_t,P_t,H_t)\rd t\right]
\\
=P_0 \int_0^\infty \rho e^{-\rho t} b(U_t,1,H_t)\rd t
+(1-P_0) \int_0^\infty \rho e^{-\rho t} b(U_t,0,H_t)\rd t.
\end{multline}
This expression is linear in $P_0$ and non-decreasing in $(P_0,H_0)$ by \autoref{ass:mon}. Taking the supremum over all controls or step controls with fixed initial condition $(P_0,H_0)$, one obtains convexity in $P_0$ and monotonicity in $(P_0,H_0)$. 
\end{proof}

\begin{lemma}[Sufficient condition for optimality of the risky arm]\label{lem:ini}
In the discretized separated problem, the risky arm is uniquely optimal as an initial choice if its expected first-stage payoff exceeds the first-stage payoff of the safe arm. 
\end{lemma}

\begin{proof}
We fix $\delta>0$ and only allow control processes which are piecewise constant on the uniform time grid of step size $\delta$. The expected first-stage payoff is denoted by
\begin{equation}\label{equ:b_bar_delta}
\ol b^\delta (u,p,h) = 
\bE\left[\int_0^\delta \rho e^{-\rho t} 
\ol b(u,P_t,H_t) \rd t\right],
\end{equation}  
where $(P,H)$ stems from a separated control with initial condition $(P_0,H_0)=(p,h)$ and constant control process $U_t\equiv u$. By Bellman's principle, optimal initial choices $U_0$ for the discretized separated problem are maximizers of 
\begin{equation}\label{equ:bellman}
\max_{u \in \bU}\ \ol b^\delta(u,p,h)
+ e^{-\rho \de } \bE\left[V^\delta(P_\delta,H_\delta)\ \middle|\ (U_0,P_0,H_0)=(u,p,h)\right].
\end{equation}
Thus, the optimal initial choice depends on the sign of the quantity
\begin{multline}\label{equ:delta}
\ol b^\delta(1,p,h) - \ol b^\delta(0,p,h) + e^{-\rho \de } 
\bE\left[V^\delta(P_\delta,H_\delta)\ \middle|\ 
(U_0,P_0,H_0)=(1,p,h)\right]
\\
- e^{-\rho \de } 
\bE\left[V^\delta(P_\delta,H_\delta)\ \middle|\ 
(U_0,P_0,H_0)=(0,p,h)\right],
\end{multline}
which is the advantage of the risky arm over the safe arm. For each $u \in \bU$, let $h_u$ be the deterministic value which $H_\delta$ attains after an initial choice of $u$. By \autoref{ass:mon}, the inequality $h_0\leq h \leq h_1$ holds. Furthermore, $P_\delta=P_0$ holds under an initial choice $u=0$. By the monotonicity and convexity result of \autoref{lem:mon},
\begin{equation}\begin{aligned}
&\bE\left[V^\de(P_\de,H_\de)\ \middle|\ (U_0,P_0,H_0)=(1,p,h)\right]
-\bE\left[V^\de(P_\de,H_\de)\ \middle|\ (U_0,P_0,H_0)=(0,p,h)\right] 
\\&\qquad=
\bE\left[V^\de(P_\de,h_1)\ \middle|\ (U_0,P_0,H_0)=(1,p,h)\right]
- V^\delta(p,h_0) 
\\&\qquad\geq
\bE\left[V^\de(P_\de,h_1)\ \middle|\ (U_0,P_0,H_0)=(1,p,h)\right]
- V^\delta(p,h_1) \geq 0.
\end{aligned}\end{equation}
It follows that \eqref{equ:delta} is strictly positive if $\ol b^\delta(1,p,h) > \ol b^\delta(0,p,h)$. In this case, the initial choice of the risky arm is uniquely optimal. 
\end{proof}

\begin{lemma}[Optimality of stopping rules]\label{lem:os}
For each $\delta>0$, $V^\delta(p,h)$ is a supremum over values of stopping rules.
\end{lemma}
  
\begin{proof}
\emph{Step 1 (Discrete setting).}
We fix $\delta>0$ and work on the uniform time grid $t_i=i\delta$, $i\in\bN$. The one-stage payoff of the problem with partial observations is given by
\begin{equation}
b^\delta(u,x,h) = \int_0^\delta \rho e^{-\rho t} b(u,x,H_t) \rd t,
\end{equation}
where $H$ stems from a control with partial observations with constant control process $U_t\equiv u$ and initial condition $(X,H_0)=(x,h)$. The one-stage payoff of the safe arm is
\begin{equation}
k^\delta = b^\delta(0,x,h)=\int_0^\delta \rho e^{-\rho t} k \rd t. 
\end{equation}
By abuse of notation, we identify indices $i\in\bN$ with times $t_i$, writing $U_i$ for the value of $U$ on $(t_i,t_{i+1}]$ and $(P_i,H_i,R_i)$ for the value of $(P,H,R)$ at $t_i$. 

\emph{Step 2 (Finite horizon).}
We truncate the problem with partial observations to a finite time horizon $n$. In the truncated problem, the value of a control $(U,X,H,R)$ is given by
\begin{equation}
J^\po(U,X,H,R) = \bE\left[\sum_{i=0}^n e^{-\rho \delta i}
b^\delta (U_i,X,H_i)\right]. 
\end{equation}
We will show by induction on $n$ that there exists an optimal stopping rule, i.e., a control that never switches from safe to the risky arm. For $n=0$, there is nothing to prove. Now let $(U,X,H,R)$ be an optimal control for the problem with horizon $n+1$ constructed via \autoref{lem:se_to_po} from an optimal Markovian control for the truncated separated problem. As $H$ evolves deterministically given $U$, it is possible to write $U=F(R)$ for a piecewise constant process $F$ on the path space $D_\R[0,\infty)$.\footnote{This is easily seen for $U_0$, which is deterministic. For $U_{i+1}$, it follows by induction because $H_{i+1}$ is a deterministic function of $U_i$.} The inductive hypothesis allows one to assume that for $i\geq 1$,   $U_i$ never switches from the safe to the risky arm. If $U_0$ indicates the risky arm, the proof is complete. Otherwise, $U$ has the form
\begin{equation}\label{equ:U_as_T}
U_i=\left\{\begin{aligned}
&0, && \text{if } i=0 \text{ or } i>T,
\\
&1, && \text{if } 1\leq i \leq T, 	
\end{aligned}\right.
\end{equation}
for some stopping time $T$. Given that the safe arm is chosen initially, the reward process $R$ is deterministic during the first stage. Therefore, there is a modification of $F$ that does not depend on the path of $R$ on the interval $[0,\delta]$. This makes it possible to define an adapted process $F^*$ which skips the first action of $F$. Then $F$ is a stopping rule. Formally, $F^*$ can be defined as 
\begin{equation}
F^*(R)=\cS^\delta F\big(\cS^{-\delta}R\big),
\end{equation}
where for any process $Z$, $(\cS^\delta Z)_t = Z_{(t+\de)\vee 0}$ is a shift of $Z$ by $\delta$. As the martingale problem $(\cA,F^*)$ is  well-posed, there is a corresponding control $(U^*,X^*,H^*,R^*)$ with $U^*=F^*(R^*)$ and initial condition $\bE[X^*]=\bE[X], H^*_0=H_0$. For comparison, we also define $(U^0,X^0,H^0,R^0)$ as the control where the risky arm $U^0\equiv 0$ is chosen all the time, still with the same initial condition $\bE[X^0]=\bE[X], H^0_0=H_0$. The values of the controls are denoted by $J,J^*$, and $J^0$, respectively. Then
\begin{align}
J^*-J^0 &=
\bE\left(\sum_{i=0}^{T-1} e^{-\rho\delta i} 
\big(b^\delta(1,X,H^*_i)-k^\delta\big)\right) 
\geq
\bE\left(\sum_{i=1}^T e^{-\rho\delta (i-1)} 
\big(b^\delta(1,X,H_i)-k^\delta\big) \right),
\\ \label{equ:step6}
J-J^0 &=
\bE \left(\sum_{i=1}^T e^{-\rho\delta i} 
\big(b^\delta(1,X,H_i)-k^\delta\big) \right) 
\geq 0.
\end{align}
The first inequality holds because choosing the safe arm decreases $H$, see \autoref{ass:mon}. The second inequality holds because the value $J$ of the optimal control is at least as high as $J^0$. Thus, 
\begin{multline}
J^*-J
\geq 
\bE\left[\sum_{i=1}^T 
\left(e^{-\rho\delta (i-1)} -e^{-\rho\delta i} \right)
\big(b^\delta(1,X,H_i)-k^\delta\big)\right]
\\
=\sum_{i=1}^\infty 
\left(e^{-\rho\delta (i-1)} -e^{-\rho\delta i} \right)
\underbrace{\bE\Big[\one_{i\leq T} 
\big(b^\delta(1,X,H_i)-k^\delta\big)\Big]}_{=:b_i}.
\end{multline}
The increments of $(b_i)_{i\in\bN}$ are given by
\begin{equation}
b_{i+1}-b_i
=
\bE\Big[\one_{i\leq T} 
\big(b^\delta(1,X,H_{i+1})-b^\delta(1,X,H_i)\big)\Big]
+\bE\Big[\one_{i=T} \big(k^\delta-b^\delta(1,X,H_{i+1})\big)\Big].
\end{equation} 
The first summand on the right-hand side is non-negative for $i\geq 1$ because $H$ increases while the risky arm is played. By the $\cF^{R}_{i+1}$-measurability of $\one_{i=T}$ and $H_{i+1}$, the second summand can be written as
\begin{equation}
\bE\Big[ \one_{i=T} \big(k^\delta-b^\delta(1,X,H_{i+1})\big)\Big]
=\bE\Big[\one_{i=T} \big(k^\delta-\ol b^\delta(1,P_{i+1},H_{i+1})\big)\Big]
=\bE\Big[\one_{i=T} \big(k^\delta-\ol b^\delta(1,P_{T+1},H_{T+1})\big)\Big],
\end{equation}
where $\ol b^\delta(u,p,h)$ is defined in \autoref{equ:b_bar_delta}. As it is optimal under $U$ (see \autoref{equ:U_as_T}) to choose the safe arm at stage $T+1$, the inequality $k^\delta\geq\ol b^\delta(1,P_{T+1},H_{T+1})$ holds by \autoref{lem:ini}. This proves $b_{i+1}\geq b_i$, for all $i\geq 1$.
By \autoref{equ:step6}, we also have 
\begin{equation}\label{equ:step6.2}
\sum_{i=1}^\infty e^{-\rho\delta i} b_i \geq 0.
\end{equation}
By \citet[Equation (5.2.8)]{Berry} this implies
\begin{equation}
J^*-J=\sum_{i=1}^\infty 
\left(e^{-\rho\delta (i-1)}-e^{-\rho\delta i}\right) b_i \geq 0,
\end{equation}
since truncated geometric discount sequences are regular. Thus we have constructed an optimal stopping rule $(U^*,X^*,H^*,R^*)$ for the truncated problem with horizon $n+1$.

\emph{Step 2 (Infinite horizon).}
We have shown that stopping rules are optimal for each discretized problem with finite horizon $n$. It follows by approximation that the value function $V^\delta(p,h)$ of the discretized problem with infinite horizon is a supremum over stopping rules. The argument can be found in the proof of \citet[Theorem 5.2.2]{Berry}. 
\end{proof}

\begin{lemma}[Description of optimal stopping rules]\label{lem:os:exi}
The stopping time $T^* = \inf \{t: V(P_t,H_t) \leq k \}$ is optimal for the separated problem.
\end{lemma}

\begin{proof}
For each $(p,h)\in[0,1]\x\bH$, there is a unique solution $(P,H,R)$ of the martingale problem for $(\cG,1)$ by \autoref{ass:wel2}. The family $(P,H)$ of processes, indexed by the initial condition $(p,h)$, is a Feller process. This follows from \citet[Theorem IX.4.39]{Jacod} using similar arguments as in Step 3 of the proof of \autoref{lem:se:app}. Let $(\wt P,\wt H)$ be the killed version of $(P,H)$ with killing rate $\rho$ and let $\Delta$ denote the ``cemetery point'' of the killed process. We refer to \citet[Section II.5.4]{Peskir} for the terminology. Let $\ol b(u,\Delta)=0$ and
\begin{equation}
A_t = A_0 + \int_0^t \big(\ol b(1,\wt P_t,\wt H_t) -k \big)\rd t.
\end{equation}
Then $Z=(\wt P,\wt H,A)$ is a Feller process on the state space $\bZ=([0,1]\x\bH \cup \{\partial\})\x\R$. Let $(\bP_z)_{z\in \bZ}$ denote the family of laws of $Z$
starting from the initial condition $Z_0=z$. There is an associated family of stopping problems
\begin{equation}
W(z) = \sup_T \bE_z(A_T),
\end{equation}
where the supremum is taken over all $\{\cF^Z_t\}$-stopping times. For any $z=(p,h,a)\neq \Delta$, 
\begin{equation}\begin{aligned}
W(z)&=\sup_T \bE_{(p,h,a)}[A_T] =\sup_T \bE_{(p,h,0)}[A_T] + a  
\\ &=
\sup_T \bE_{(p,h,0)}\left[\int_0^T \rho e^{-\rho t} 
\big(\ol b(1,P_t,H_t)-k\big) \rd t\right] +a 
= V(p,h)-k+a,
\end{aligned}\end{equation}
because $V(p,y)$ is a supremum of values of stopping rules by part (a) of \autoref{thm:main}. The stopping set $\bD\subset \bZ$ is defined as in \citet[Equation (2.2.5)]{Peskir} by
\begin{equation}\label{equ:sto_D}
\bD=\big\{z=(p,h,a) \in \bZ\colon W(z) \leq a \big\} 
=\Big(\big\{(p,h)\in [0,1]\x\bH\colon V(p,h) \leq k \big\} \cup \{\Delta\}\Big)\x \R.
\end{equation}
The last equality holds because $W(\partial,a)=a$ by definition. The function $W$ is lower semi-continuous by \citet[Equation (2.2.80)]{Peskir} because $(\wt P,\wt H,A)$ is Feller. Therefore, the set $\bD$ is closed. Then the right-continuity of the filtration implies that 
  \begin{equation}
    T^* = \inf \{t\geq0: X_t \in D\}= \inf \{t: V(P_t,H_t) \leq k \}
  \end{equation}
  is a stopping time. Note that $\Delta \in D$, which implies $\bP(T^*<\infty)=1$. 
  Then \citet[Corollary 2.9]{Peskir} implies that $T^*$ is optimal. 
\end{proof}

\begin{lemma}[Asymptotic learning]\label{lem:asy}
Assume $0<P_0<1$. Then the following statements hold for any control $(U,X,H,R)$ of the problem with partial observations and the corresponding belief process $P$.
\begin{enumerate}[(a)]
	\item Assume that the measures $K_R(1,h,\cdot)$ and $K_R(0,h,\cdot)$ are equivalent for all $h$. Then learning in finite time is impossible, i.e., 
	$0<P_t<1$ holds a.s. for all $t\geq 0$. Moreover, asymptotic learning does not occur if the agent invests only a finite amount of time into the risky arm, i.e.,
	\begin{equation}
		\{\textstyle\int_0^\infty U_t dt < \infty\} \subseteq \{0<P_\infty<1\} \quad \bP\text{-a.s.}
	\end{equation}
	\item Assume that $\Phi(1,\cdot)$ is bounded from below by a positive constant. Then asymptotic learning is guaranteed if the agent invests an infinite amount of time in the risky arm, i.e., 
	\begin{equation}
		\{\textstyle\int_0^\infty U_t dt = \infty\} \subseteq \{P_\infty = X\} \quad \bP\text{-a.s.}
	\end{equation}  
	\item If the conditions of (a) and (b) are satisfied, then asymptotic learning occurs if and only if the agent invests an infinite amount of time in the risky arm: 
	\begin{equation}
		\{\textstyle\int_0^\infty U_t dt = \infty\} = \{P_\infty = X\} \quad \bP\text{-a.s.}
	\end{equation}
\end{enumerate}
\end{lemma}

\begin{proof}
\emph{Step 1 (Hellinger process).} 
Let $\bP_1$ and $\bP_0$ be defined by conditioning the measure $\bP$ on the events $X=1$ and $X=0$, respectively.  We want to calculate the Hellinger process $h(\tfrac12)$ of order $\tfrac12$ of the measures $\bP_1$ and $\bP_0$. Let $P_t=\bE[X\mid\cF^Y_t]$ be the belief process. By \autoref{equ:P_via_D}, $P/P_0$ is the density process of $\bP_1$ relative to $\bP$. Similarly, $(1-P)/(1-P_0)$ is the density process of $\bP_0$ relative to $\bP$. For all $p,q \in \R$, let 
\begin{equation}
\psi(p,q)=\frac{p+q}{2}-\sqrt{pq}
\end{equation}
and let $\nu(\rd t,\rd p,\rd q)$ be the compensator of the integer-valued random measure associated the jumps of $(P,1-P)$. Let $S$ be the first time that $P$ or $P_-$ hits zero or one, 
\begin{equation}\label{equ:hell_S}
S= \inf \big\{t \geq 0: P_t\in \{0,1\} 
\text{ or } P_{t-} \in \{0,1\}\big\}.
\end{equation}
By \citet[Lemma III.3.7]{Jacod}, $P$ is constant on $\llbracket S,\infty \llbracket$. Therefore, on this interval, $\langle P^c,P^c \rangle$ is constant and $\nu$ has no charge. After canceling out the terms $P_0$ and $(1-P_0)$, the formula for $h(\tfrac12)$ given in \citet[Theorem IV.1.33]{Jacod} reads as
\begin{equation}\label{equ:h}\begin{aligned}
h(\tfrac12) 
&= 
\frac18 \left( \frac{1}{P_-^2} \bullet \langle P^c,P^c \rangle 
- \frac{2}{P_-(1-P_-)} \bullet \langle P^c,1-P^c \rangle 
+ \frac{1}{(1-P_-)^2} \bullet \langle 1-P^c,1-P^c \rangle \right) 
\\ &\qquad
+  \psi\left(1+\frac{p}{P_-},1+\frac{q}{1-P_-}\right) 
* \nu(\rd t,\rd p,\rd q) 
\\&=
\frac18 \left( \frac{1}{P_-} + \frac{1}{1-P_-} \right) 
\bullet \langle P^c,P^c \rangle 
\\&\qquad
+  \psi\left(1+\frac{j(U,P_-,Y_-,z)}{P_-},1-\frac{j(U,P_-,Y_-,z)}{1-P_-}\right) 
* \ol K(U,P_-,Y_-,\rd z)\rd t
\\&=
\frac18 \phi_1(U,Y)^\top\si^2(U,Y)\phi_1(U,Y) \bullet I^S 
\\&\qquad
+  \psi\Bigg(\frac{\phi_2(U,Y_-,z)}
{P_- \phi_2(U,Y_-,z)+(1-P_-)\big(2-\phi_2(U,Y_-,z)\big)},
\\&\qquad\qquad\qquad
\frac{2-\phi_2(U,Y_-,z)}
{P_- \phi_2(U,Y_-,z)+(1-P_-)\big(2-\phi_2(U,Y_-,z)\big)}\Bigg) 
\one_{\llbracket 0,S \rrbracket}* \ol K(U,P_-,Y_-,\rd z)\rd t
\\&=
\frac18 \phi_1(U,Y)^\top\si^2(U,Y)\phi_1(U,Y) \bullet I^S 
\\&\qquad
+  \int \frac{1-\sqrt{\phi_2(U,Y,z)\big(2-\phi_2(U,Y,z)\big)}}
{P \phi_2(U,Y,z)+(1-P)\big(2-\phi_2(U,Y,z)\big)}
\ol K(U,P,Y,\rd z) \bullet I^S
\\&=
\frac18 \phi_1(U,Y)^\top\si^2(U,Y)\phi_1(U,Y) \bullet I^S 
\\&\qquad
+  \int \left(1-\sqrt{\phi_2(U,Y,z)\big(2-\phi_2(U,Y,z)\big)}\right) 
\ol K(U,1/2,Y,\rd z) \bullet I^S 
= \Phi(U,Y) \bullet I^S,
\end{aligned}\end{equation} 
where $\Phi$ is defined in \autoref{ass:novikov}.

\emph{Step 2 (Finite investment prevents asymptotic learning).} 
We define stopping times $T$ and $T_n$ as in \autoref{equ:T,equ:T_n}. $T$ is the first time that $P$ or $P_-$ hits zero and $T_n$ announces $T$. Let us assume for contradiction that $P$ jumps to zero, i.e., $P_{T-}>0$. Then $T_n=T$ holds for all sufficiently large $n$. Consequently, the process $D^n=\mathcal E(L^n)=P^{T_n}/P_0$ defined in \autoref{equ:Ln} also jumps to zero. Therefore, $L^n$ has a jump of height $-1$. This is not possible because $\phi_2(u,y,z)>0$ holds by the assumption that $K(u,1,y,\cdot)$ and $K(u,0,y,\cdot)$ are equivalent. This proves that $P$ does not jump to zero. A similar argument where the r\^oles of $\bP_0$ and $\bP_1$ are reversed shows that $P$ cannot jump to one. It follows that for any stopping time $\tau$, 
\begin{alignat}{5}
  	\label{equ:hell1}&\{h(\tfrac12)_\tau=\infty\}&&=\{S\leq\tau,P_{S-}=0\}
      &&=\{P_\tau=0\}&&=\{P_\tau=0 \text{ or } P_\tau=1 \} &\quad& \bP_0\text{-a.s.,}\\
    \label{equ:hell2}&\{h(\tfrac12)_\tau=\infty\}&&=\{S\leq\tau,P_{S-}=1\}
      &&=\{P_\tau=1\}&&=\{P_\tau=0 \text{ or } P_\tau=1 \} &\quad& \bP_1\text{-a.s.}
\end{alignat}
In \autoref{equ:hell1,equ:hell2}, the first equality holds by \citet[Theorem 1.5]{Schachermayer}. This theorem states that the divergence of the Hellinger process is equivalent to the mutual singularity of the measures $\bP_1$ and $\bP_0$, but in such a way that the singularity is not obtained by a sudden jump of the density process to zero or one. The second equality holds because such jumps are not possible by the previous claim. For the third equality, see \citet[Proposition III.3.5.(ii)]{Jacod}. By \autoref{ass:str}, the safe arm reveals no information about the hidden state $X$, resulting in $\Phi(0,y)=0$. Together with \autoref{ass:novikov} bounding $\Phi$ from above, \autoref{equ:hell1,equ:hell2} imply
\begin{equation}
  	\{\textstyle\int_0^\infty U_t dt < \infty \} \subseteq \{h(\tfrac12)_\infty < \infty \} = \{0<P_\infty<1\} \quad \bP\text{-a.s.}
\end{equation}
  This proves (a).

\emph{Step 3 (Infinite investment induces asymptotic learning).}
Let $\tau$ be a stopping time. If $S$ does not occur before $\tau$ and $\int_0^\tau U_t dt=\infty$, then $h(\tfrac12)_\tau=\infty$ because of the lower bound $\inf_y \Phi(1,y)>0$. Therefore,
\begin{equation}
\{\textstyle\int_0^\tau U_t dt = \infty \} \subseteq \{ h(\tfrac12)_\tau<\infty\} \cup \{S \leq \tau\}.
\end{equation}
Moreover, it follows from \citet[Theorem 1.5]{Schachermayer} that
\begin{alignat}{4}
    &\{ h(\tfrac12)_\tau<\infty\} \cup \{S \leq \tau\}
      &&=\{S\leq\tau,P_{S-}=0\}\cup\{S\leq\tau\}=\{P_\tau=X\} &\quad&\bP_0\text{-a.s.,}\\
    &\{ h(\tfrac12)_\tau<\infty\} \cup \{S \leq \tau\}
      &&=\{S\leq\tau,P_{S-}=1\}\cup\{S\leq\tau\}=\{P_\tau=X\} &\quad&\bP_1\text{-a.s.}
\end{alignat}
It follows that
\begin{equation}
\{\textstyle\int_0^\tau U_t dt = \infty \} \subseteq \{P_\tau=X\} \quad \bP\text{-a.s.,}
\end{equation}
which proves (b). Finally, (c) follows from (a) and (b). 
\end{proof}

\printbibliography
\end{document}